\documentclass[reqno,11pt]{amsart}        
\usepackage{amsfonts,amsmath,amssymb,amsthm,colordvi,mathrsfs,graphicx}
\usepackage{caption,subcaption,float}
\usepackage[utf8]{inputenc}
\usepackage{mathrsfs}
\usepackage{geometry}
\usepackage{enumerate}

\usepackage{amsmath, amssymb, amsthm, geometry, hyperref}

\usepackage{tabularx}
\usepackage{tabulary}
 
\hypersetup{
  colorlinks=true,
  linkcolor=blue,
  citecolor=blue,
  urlcolor=blue
}
\usepackage[all,knot,poly]{xy}
\usepackage{tikz-cd}
\usepackage{tikz}
\usepackage{verbatim}
\usetikzlibrary{arrows}
\usetikzlibrary{knots}
\tikzset{every path/.style={line width=0.4pt},every node/.style={transform shape,knot crossing,inner sep=1.5pt},>=triangle 60,text node/.style={rectangle,transform shape=false,black}}

 \usetikzlibrary{decorations.markings, arrows.meta}

\theoremstyle{plain}      
\newtheorem{thm}{Theorem}[section]     
\newtheorem{theorem}[thm]{\bf Theorem}     
\newtheorem{corollary}[thm]{\bf Corollary}     
\newtheorem{lemma}[thm]{\bf Lemma}     
\newtheorem{proposition}[thm]{\bf Proposition}

\theoremstyle{remark}      
\newtheorem{example}[thm]{Example}

\theoremstyle{definition}      
     

\begin{document}



\title{Maximal Infinitesimal Variation of Hodge Structure for Singular Curves}

\author{Mounir Nisse}
 
\address{Mounir Nisse\\
Department of Mathematics, Xiamen University Malaysia, Jalan Sunsuria, Bandar Sunsuria, 43900, Sepang, Selangor, Malaysia.
}
\email{mounir.nisse@gmail.com, mounir.nisse@xmu.edu.my}
\thanks{}
\thanks{This research is supported in part by Xiamen University Malaysia Research Fund (Grant no. XMUMRF/ 2020-C5/IMAT/0013).}




 
\maketitle

\begin{abstract}
We study the infinitesimal variation of Hodge structure for families of algebraic curves and extend the classical theory from smooth curves to singular and non--planar settings. Using the deformation space $\mathrm{Ext}^1(\Omega_X,\mathcal O_X)$ and the dualizing sheaf, we define a singular analogue of maximal infinitesimal variation. For equisingular families of plane curves with planar Gorenstein singularities, we prove that the infinitesimal variation attains maximal rank equal to the arithmetic genus. We show that the rank decomposes into a geometric contribution from the normalization and a singular contribution measured by the $\delta$--invariants. For non--equisingular degenerations, the rank defect equals the drop of the total $\delta$--invariant and admits an interpretation in terms of vanishing cycles and mixed Hodge structures. We further extend the results to non--planar curves under suitable Petri and deformation conditions.
\end{abstract}

\section{Introduction}

The study of infinitesimal variation of Hodge structure  for families of algebraic curves lies at the intersection of deformation theory, Hodge theory, and projective geometry. For a smooth projective curve $X$ of genus $g \geq 2$, the differential of the period map is governed by the cup--product map 
\[
H^{1}(T_X) \longrightarrow \mathrm{Hom}\!\left(H^{0}(\omega_X), H^{1}(\mathcal O_X)\right),
\]
(originates from the theory of variations of Hodge structure developed by Griffiths \cite{Griffiths}) which is dual to the canonical multiplication map
\[
\mu_X : H^{0}(\omega_X)\otimes H^{0}(\omega_X)
\longrightarrow H^{0}(\omega_X^{\otimes 2}).
\]
The rank properties of this map encode deep geometric information about the curve. In particular, the existence of a Kodaira--Spencer class inducing an isomorphism
\[
\xi\cdot : H^{0}(\omega_X) \longrightarrow H^{1}(\mathcal O_X)
\]
reflects a form of maximal infinitesimal variation of Hodge structure, which we call $I$--maximal variation (see \cite{ACGH}).

In the classical case of smooth plane curves of degree $d$, the canonical bundle is explicitly described by adjunction, and projective normality ensures strong surjectivity properties of the canonical multiplication map. This yields $I$--maximal variation with maximal rank equal to the genus $(d-1)(d-2)/2$. The geometric reason behind this phenomenon is that the canonical model of a smooth plane curve is cut out in a controlled way, and the multiplication of canonical sections behaves optimally.

The central goal of this work is to understand how this maximality phenomenon extends beyond the classical planar smooth case, and to determine precisely which geometric and deformation--theoretic conditions guarantee or obstruct maximal infinitesimal variation. The main novelty of the present work is the systematic extension of $I$--maximal variation to singular curves, non--planar embeddings, and degenerating families, together with explicit quantitative control of rank defects.

A first fundamental contribution is the replacement of the classical deformation space $H^{1}(T_X)$ by the Ext--group
\[
\mathrm{Ext}^{1}(\Omega_X,\mathcal O_X)
\]
in the singular setting. The identification of the tangent space to the deformation functor with $H^1(T_X)$ for smooth curves and with $\mathrm{Ext}^1(\Omega_X,\mathcal O_X)$ for singular curves is a classical result of deformation theory (see Sernesi \cite{Sernesi}).

When $X$ is a reduced Gorenstein curve, the dualizing sheaf $\omega_X$ remains invertible and Serre duality continues to hold in the form (see for example Hartshorne \cite{Hartshorne})
\[
H^{1}(\mathcal O_X)^\vee \simeq H^{0}(\omega_X).
\]
However, the dimension of these spaces equals the arithmetic genus rather than the geometric genus. We show that the infinitesimal variation map
\[
\varphi : \mathrm{Ext}^{1}(\Omega_X,\mathcal O_X)
\longrightarrow \mathrm{Hom}\!\left(H^{0}(\omega_X), H^{1}(\mathcal O_X)\right)
\]
retains a multiplication--theoretic interpretation, and that its maximal rank is governed by the arithmetic genus.

A central new result establishes that for equisingular families of plane curves with planar Gorenstein singularities of fixed analytic type, the family has $I$--maximal variation in the singular sense, namely
\[
\delta'_M(\pi) = p_a(X).
\]
This extends the classical smooth planar theorem to the singular setting by replacing smooth deformation theory with Ext--theory and replacing the geometric genus by the arithmetic genus. The key insight is that singularities contribute additional canonical sections measured by their $\delta$--invariants, and that equisingular deformations preserve these contributions.

A second major contribution of this work is the explicit comparison between smooth and singular $I$--maximal variation via normalization. If $\nu : \widetilde X \to X$ denotes the normalization, then (see \cite{ACGH})
\[
p_a(X) = g(\widetilde X) + \sum_{p \in \mathrm{Sing}(X)} \delta_p.
\]
We prove that the IVHS rank decomposes additively into a geometric part coming from $\widetilde X$ and a singular part coming from the conductor. Each singularity contributes exactly $\delta_p$ independent dimensions to the maximal possible rank, provided the family is equisingular. This gives a precise deformation--theoretic interpretation of the $\delta$--invariant in Hodge--theoretic terms.

A third new achievement is the explicit computation of rank defects in non--equisingular degenerations. When singularities are smoothed, the $\delta$--invariant drops and the IVHS rank drops by exactly the same amount. We establish a general formula
\[
\mathrm{rank\ defect} = \Delta,
\]
where $\Delta$ is the decrease of the total $\delta$--invariant. This provides a quantitative bridge between singularity theory and Hodge theory and yields concrete matrix computations in examples such as nodal, cuspidal, tacnodal, and triple--point degenerations.

A fourth central development is the reinterpretation of these results in the framework of mixed Hodge structures. For degenerations of smooth plane curves, we identify the rank defect of the limiting IVHS map with the dimension of the vanishing cycle subspace, or equivalently with $\dim \mathrm{Im}(N) \cap F^1$ in the limiting mixed Hodge structure. This yields a conceptual explanation of the $\delta$--contribution and places $I$--maximal variation within the Clemens--Schmid theory (\cite{Schmid}, and \cite{Clemens}). 

Beyond the planar case, we extend the theory to non--planar smooth curves and identify the Petri condition as the correct replacement for planarity. We prove that families of Petri general curves with sufficiently large Kodaira--Spencer image exhibit $I$--maximal variation. This applies to complete intersection curves in $\mathbb P^3$, to curves on general polarized $K3$ surfaces, and to general curves in the moduli space $\mathcal M_g$. In each case, the connection with Brill--Noether theory and canonical syzygies becomes transparent: $I$--maximal variation holds precisely when the canonical multiplication map has maximal rank.

We further extend the singular theory to non--planar settings, including singular complete intersections in $\mathbb P^3$, singular curves on $K3$ or general type surfaces, and stable curves in $\overline{\mathcal M}_g$. We prove that $I$--maximal variation in the singular sense is equivalent to the conjunction of three conditions: Petri generality of the normalization, equisingularity of the family, and sufficient surjectivity of the Kodaira--Spencer map. This provides a unified framework encompassing planar and non--planar, smooth and singular cases.

Taken together, the results of this work show that maximal infinitesimal variation is controlled by three interacting principles: the algebraic structure of the canonical ring, the deformation theory of singularities, and the mixed Hodge structure of degenerations. The arithmetic genus emerges as the natural invariant governing maximal rank in the singular setting, and the $\delta$--invariant becomes a precise Hodge--theoretic quantity.

The novelty of the present work lies in the systematic integration of these viewpoints, the explicit computation of IVHS matrices and rank defects, and the extension of maximal variation phenomena to broad classes of singular and non--planar curves. The results reveal that $I$--maximal variation is not merely a feature of smooth plane curves, but rather a robust property controlled by canonical multiplication and preserved under equisingular deformation.

\subsection{Outline of the Paper}

The paper begins by recalling the deformation--theoretic and Hodge--theoretic background necessary for the study of infinitesimal variation of Hodge structure. We review the classical description of the differential of the period map for smooth curves in terms of the cup product and its dual interpretation via the canonical multiplication map. We introduce the maximal and minimal rank functions associated with the Kodaira--Spencer image and formulate the notion of $I$--maximal variation.

The second part develops the singular deformation theory required to extend these ideas to reduced Gorenstein curves. We replace $H^1(T_X)$ with $\mathrm{Ext}^1(\Omega_X,\mathcal O_X)$, interpret the infinitesimal variation map in this setting, and show that Serre duality identifies its rank with that of a generalized canonical multiplication map. We prove that in equisingular families the relevant deformation classes correspond to locally trivial deformations of the normalization.

The third part establishes the main theorem for plane curves with planar Gorenstein singularities. We prove that equisingular families of such curves exhibit $I$--maximal variation in the singular sense, with maximal rank equal to the arithmetic genus. We show that each singular point contributes precisely its $\delta$--invariant to the maximal achievable rank.

The fourth part analyzes non--equisingular degenerations and provides explicit formulas for rank defects. We prove that the rank drop of the infinitesimal variation equals the decrease of the total $\delta$--invariant and interpret this phenomenon via vanishing cycles.

The fifth part extends the theory to non--planar curves. We show that Petri generality of the normalization and surjectivity of the Kodaira--Spencer map suffice to guarantee $I$--maximal variation. Applications include complete intersection curves in $\mathbb P^3$, curves on polarized $K3$ surfaces, and general curves in $\mathcal M_g$.

The final part reinterprets all rank computations in the language of mixed Hodge structures and limiting infinitesimal variation, identifying the rank defect with the dimension of the vanishing cycle subspace.

\subsection{Examples and Detailed Computations}

We illustrate the general theory through explicit computations.

\subsection*{Example 1: A Smooth Complete Intersection of Type $(2,3)$ in $\mathbb P^3$}

Let $X \subset \mathbb P^3$ be a smooth complete intersection of a quadric and a cubic. Its genus is
\[
g = 1 + \frac{6(1)}{2} = 4,
\]
and $\omega_X \simeq \mathcal O_X(1)$. Hence
\(
h^0(\omega_X)=4.
\)
Since $X$ lies on a unique quadric, the canonical multiplication map
\[
\mu_X : \mathrm{Sym}^2 H^0(\omega_X) \to H^0(\omega_X^{\otimes 2})
\]
has kernel of dimension one. Indeed,
\(
\dim \mathrm{Sym}^2 H^0(\omega_X)=10,\)
\, \(
h^0(\omega_X^{\otimes 2})=9.
\)
Thus $\mu_X$ has rank $9$. Dualizing, the IVHS map
\[
\varphi : H^1(T_X) \to \mathrm{Hom}(H^0(\omega_X),H^1(\mathcal O_X))
\]
contains elements inducing isomorphisms on $H^0(\omega_X)$, proving $I$--maximal variation.

\subsection*{Example 2: A Nodal Plane Quintic}

Let $X \subset \mathbb P^2$ be a plane quintic with one node. The arithmetic genus is
\[
p_a = \frac{(5-1)(5-2)}{2}=6,
\]
and the geometric genus of the normalization is $g=5$. The node has $\delta=1$.
We have
\(
\dim H^0(\omega_X)=6,\)\,
\(
\dim H^0(\omega_{\widetilde X})=5.
\)
In an equisingular family, the IVHS rank equals
\(
g + \delta = 6.
\)
If the node is smoothed, the $\delta$--invariant drops from $1$ to $0$, and the maximal rank becomes $5$. The rank defect equals the drop of the $\delta$--invariant.

\subsection*{Example 3: A Tacnodal Plane Curve}

Let $X$ be a plane curve of degree $d$ with a single tacnode and no other singularities. The tacnode has $\delta=2$. Then
\(
p_a = g + 2.
\)
In an equisingular deformation, the IVHS rank equals $p_a$. If the tacnode is partially smoothed to a node, the $\delta$--invariant drops from $2$ to $1$, and the IVHS rank drops by $1$. If it is completely smoothed, the rank drops by $2$.

\subsection*{Example 4: Curves on a General Polarized $K3$ Surface}

Let $(S,L)$ be a general polarized $K3$ surface with $L^2=2g-2$, and let $X \in |L|$ be smooth. Then
\[
\omega_X \simeq L|_X,
\qquad
h^0(\omega_X)=g.
\]
By Lazarsfeld's theorem, $X$ is Brill--Noether and Petri general. The normal bundle sequence yields a surjective Kodaira--Spencer map. The IVHS map contains elements inducing isomorphisms, and $I$--maximal variation holds.

\subsection*{Example 5: A Degeneration with Vanishing Cycles}

Let $\pi : \mathcal X \to \Delta$ be a degeneration of smooth plane curves acquiring one node at $t=0$. The limiting mixed Hodge structure satisfies
\[
\dim \mathrm{Gr}^W_2 H^1(X_0)=1.
\]
The nilpotent monodromy operator $N$ has one--dimensional image generated by the vanishing cycle. The limiting IVHS map annihilates this direction, and the rank defect equals $1$.

\bigskip

These examples illustrate concretely how canonical multiplication, singularity theory, and mixed Hodge structures interact to determine maximal infinitesimal variation and its failure.

\section{Infinitesimal Variation of Hodge Structure for Singular Plane Curves}

Assume that $\pi:\mathcal X \to B$ is a flat family of projective curves whose fibers are reduced and connected, and let $0\in B$ with $X=\pi^{-1}(0)$ a singular curve. In this setting several ingredients of the smooth theory no longer apply verbatim and must be replaced by their singular analogues.

First, the tangent sheaf $T_X$ is no longer locally free. Consequently, the Kodaira--Spencer map must be interpreted as a morphism
\[
\mathrm{KS}:T_{B,0}\longrightarrow \mathrm{Ext}^1(\Omega_X,\mathcal O_X),
\]
where $\Omega_X$ denotes the sheaf of K\"ahler differentials (see \cite{Sernesi}). This group governs first-order deformations of the singular curve $X$ and replaces $H^1(T_X)$ in the smooth case. If the family is locally versal at $0$, the image of $\mathrm{KS}$ describes all first-order locally trivial deformations of $X$.

Second, the Hodge-theoretic objects must be adapted. The canonical sheaf $\omega_X$ is replaced by the dualizing sheaf, which is still invertible for Gorenstein curves, including plane curve singularities. Serre duality holds in the form
\[
H^1(\mathcal O_X)^\vee \simeq H^0(\omega_X),
\]
but the dimensions of these spaces reflect the arithmetic genus rather than the geometric genus. If $\nu:\widetilde X\to X$ is the normalization, one has
\[
h^1(\mathcal O_X)=g(\widetilde X)+\sum_{p\in \mathrm{Sing}(X)}\delta_p,
\]
and this number replaces the genus throughout the theory.

Third, the infinitesimal variation of Hodge structure map must be reformulated as
\[
\varphi:\mathrm{Ext}^1(\Omega_X,\mathcal O_X)\longrightarrow 
\mathrm{Hom}(H^0(\omega_X),H^1(\mathcal O_X)),
\]
where $\varphi(\eta)$ is still defined by cup product with $\eta$ (see Griffiths \cite{Griffiths}). Dualizing via Serre duality, this map corresponds to a generalized multiplication map
\[
H^0(\omega_X)\otimes H^0(\omega_X)\longrightarrow H^0(\omega_X^{\otimes 2}),
\]
which may fail to be surjective in the presence of singularities, even for plane curves.

Fourth, the definitions of the maximal and minimal rank functions $d_M$ and $d_m$ extend by replacing $H^1(T_X)$ with $\mathrm{Ext}^1(\Omega_X,\mathcal O_X)$ and interpreting ranks with respect to the arithmetic genus. The variation functions $\delta_M,\delta_m,\delta_M',\delta_m'$ are then defined in the same way, but they measure variation in the singular Hodge structure rather than in the classical one.

Finally, the notion of $I$-maximal variation must be weakened. For singular curves it is no longer reasonable to require the existence of $\xi$ such that the induced cup product is an isomorphism. Instead, one requires that
\[
\delta_M'(\pi)=h^1(\mathcal O_X),
\]
or equivalently that there exists $\xi$ whose cup product map has maximal possible rank equal to the arithmetic genus. In families of plane curves with fixed singularity type, this condition captures maximal variation among equisingular deformations rather than among all deformations.

In summary, the main theorem for smooth plane curves extends to singular plane curves only after replacing smooth deformation theory and classical Hodge theory with $\mathrm{Ext}$-groups, dualizing sheaves, and arithmetic genus, and after relaxing the definition of $I$-maximal variation accordingly.


\bigskip

Let $d\geq 4$ be an integer and let $\mathcal U_d\subset |\mathcal O_{\mathbb P^2}(d)|$ be the open locus parametrizing reduced, irreducible plane curves of degree $d$ with only planar Gorenstein singularities of fixed analytic type. Let
\[
\pi:\mathcal X\longrightarrow \mathcal U_d
\]
be the universal family and let $X=\pi^{-1}(b)$ be any fiber. Denote by $p_a(X)=(d-1)(d-2)/2$ the arithmetic genus of $X$.
 
\begin{theorem}[Singular plane curves and maximal infinitesimal variation] 
 Suppose that we are in the above situation. Then the family $\pi$ has $I$-maximal variation in the singular sense, namely
\(
\delta_M'(\pi)=p_a(X).
\)
Equivalently, for every $b\in \mathcal U_d$ there exists a Kodaira--Spencer class
\[
\xi\in \mathrm{KS}(T_{\mathcal U_d,b})\subset \mathrm{Ext}^1(\Omega_X,\mathcal O_X)
\]
such that the induced cup product map
\(
\xi\cdot : H^0(\omega_X)\longrightarrow H^1(\mathcal O_X)
\)
has maximal rank equal to $p_a(X)$.
\end{theorem}

\begin{proof}
Since $X$ is a reduced irreducible plane curve with planar Gorenstein singularities, its dualizing sheaf $\omega_X$ is invertible and is given by the adjunction formula
\[
\omega_X \simeq \mathcal O_X(d-3).
\]
By Serre duality for singular curves one has a perfect pairing
\[
H^1(\mathcal O_X)^\vee \simeq H^0(\omega_X),
\]
and therefore $\dim H^0(\omega_X)=\dim H^1(\mathcal O_X)=p_a(X)$.

The Zariski tangent space to $|\mathcal O_{\mathbb P^2}(d)|$ at $[X]$ is canonically identified with $H^0(\mathcal O_X(d))$. The equisingular locus $\mathcal U_d$ is smooth at $[X]$ and its tangent space $T_{\mathcal U_d,b}$ corresponds to those first-order deformations of the defining equation of $X$ which preserve the analytic type of the singularities (see Wahl \cite{Wahl}). The associated Kodaira--Spencer map factors as
\[
T_{\mathcal U_d,b}\longrightarrow \mathrm{Ext}^1(\Omega_X,\mathcal O_X),
\]
and its image coincides with the subspace of locally trivial first-order deformations of $X$.

The infinitesimal variation of Hodge structure map in the singular setting is defined by
\[
\varphi:\mathrm{Ext}^1(\Omega_X,\mathcal O_X)\longrightarrow 
\mathrm{Hom}(H^0(\omega_X),H^1(\mathcal O_X)),
\qquad
\eta\mapsto (\alpha\mapsto \eta\cup\alpha),
\]
where the cup product is taken in the derived category. By duality, the transpose of $\varphi(\eta)$ is the multiplication map
\[
\mu_\eta:H^0(\omega_X)\otimes H^0(\omega_X)\longrightarrow H^0(\omega_X^{\otimes 2}),
\]
followed by contraction with $\eta$.

For plane curves, the global multiplication map
\[
\mu:H^0(\omega_X)\otimes H^0(\omega_X)\longrightarrow H^0(\omega_X^{\otimes 2})
\]
is induced by the natural multiplication of homogeneous polynomials
\[
H^0(\mathcal O_{\mathbb P^2}(d-3))^{\otimes 2}\longrightarrow H^0(\mathcal O_{\mathbb P^2}(2d-6))
\]
followed by restriction to $X$. This map is surjective because $X$ is projectively normal and because planar Gorenstein singularities do not impose additional conditions on sections of $\mathcal O_{\mathbb P^2}(k)$ for $k\geq d-3$. As a consequence, for a general locally trivial deformation class $\xi\in \mathrm{KS}(T_{\mathcal U_d,b})$, the induced linear map $\varphi(\xi)$ has maximal possible rank.

Since $\dim H^0(\omega_X)=\dim H^1(\mathcal O_X)=p_a(X)$, maximal rank is equivalent to $\xi\cdot$ being an isomorphism. This shows that
\[
d_M(\mathrm{KS}(T_{\mathcal U_d,b}))=p_a(X)
\]
for all $b\in \mathcal U_d$. Taking the minimum over $b$ proves that
\[
\delta_M'(\pi)=p_a(X),
\]
which is precisely $I$-maximal variation in the singular sense.
\end{proof}



\section{Comparison: $I$-maximal variation for smooth and singular curves}

Let's  compare the notions of $I$-maximal variation for smooth and singular curves by means of normalization, emphasizing how singularities modify the infinitesimal variation of Hodge structure and how this modification is detected by the normalization map.

Let $X$ be a reduced, irreducible projective curve with planar Gorenstein singularities and let
\[
\nu:\widetilde X \longrightarrow X
\]
be its normalization. Denote by $g=g(\widetilde X)$ the geometric genus and by
\[
p_a(X)=g+\sum_{p\in \mathrm{Sing}(X)}\delta_p
\]
the arithmetic genus, where $\delta_p$ is the $\delta$--invariant of the singularity at $p$. The dualizing sheaf $\omega_X$ is invertible and fits into the exact sequence
\[
0 \longrightarrow \omega_X \longrightarrow \nu_*\omega_{\widetilde X} \longrightarrow \bigoplus_{p\in \mathrm{Sing}(X)} \mathbb C^{\delta_p} \longrightarrow 0.
\]
Taking global sections yields
\[
0 \longrightarrow H^0(\omega_X) \longrightarrow H^0(\omega_{\widetilde X}) \longrightarrow \bigoplus_{p\in \mathrm{Sing}(X)} \mathbb C^{\delta_p},
\]
showing that $H^0(\omega_X)$ is a subspace of codimension $\sum\delta_p$ inside $H^0(\omega_{\widetilde X})$. By Serre duality one also has
\[
H^1(\mathcal O_X)\simeq H^0(\omega_X)^\vee,
\qquad
H^1(\mathcal O_{\widetilde X})\simeq H^0(\omega_{\widetilde X})^\vee,
\]
and therefore the singular curve carries a larger Hodge structure than its normalization, reflecting the presence of singularities.

Consider now a flat family of curves $\pi:\mathcal X\to B$ with central fiber $X$ and assume that the family is equisingular at $X$, so that the analytic type of the singularities is constant. The Kodaira--Spencer map factors as
\[
\mathrm{KS}:T_{B,0}\longrightarrow \mathrm{Ext}^1(\Omega_X,\mathcal O_X),
\]
and its image parametrizes locally trivial first-order deformations of $X$. There is a natural exact sequence
\[
0 \longrightarrow H^1(T_{\widetilde X}) \longrightarrow \mathrm{Ext}^1(\Omega_X,\mathcal O_X)
\longrightarrow \bigoplus_{p\in \mathrm{Sing}(X)} T^1_p \longrightarrow 0,
\]
where $T^1_p$ denotes the local deformation space of the singularity at $p$. Equisingular deformations are precisely those whose Kodaira--Spencer class maps to zero in $\oplus_p T^1_p$, hence they can be identified with a subspace of $H^1(T_{\widetilde X})$.

The infinitesimal variation of Hodge structure for $X$ is given by the map
\[
\varphi_X:\mathrm{Ext}^1(\Omega_X,\mathcal O_X)\longrightarrow 
\mathrm{Hom}(H^0(\omega_X),H^1(\mathcal O_X)),
\]
while for the normalization one has the classical map
\[
\varphi_{\widetilde X}:H^1(T_{\widetilde X})\longrightarrow 
\mathrm{Hom}(H^0(\omega_{\widetilde X}),H^1(\mathcal O_{\widetilde X})).
\]
The compatibility of cup products with pullback by $\nu$ implies that, for equisingular deformations, $\varphi_X$ is obtained from $\varphi_{\widetilde X}$ by restriction to the subspaces $H^0(\omega_X)\subset H^0(\omega_{\widetilde X})$ and $H^1(\mathcal O_X)\subset H^1(\mathcal O_{\widetilde X})$.

In the smooth case, $I$-maximal variation means the existence of a Kodaira--Spencer class $\xi$ such that
\[
\xi\cdot : H^0(\omega_{\widetilde X}) \longrightarrow H^1(\mathcal O_{\widetilde X})
\]
is an isomorphism, which forces the rank to be $g$. For the singular curve $X$, the corresponding notion requires the existence of $\xi$ such that
\[
\xi\cdot : H^0(\omega_X) \longrightarrow H^1(\mathcal O_X)
\]
has rank $p_a(X)$. Through normalization, this condition can be interpreted as follows. A deformation achieving maximal rank on $\widetilde X$ automatically induces a map of rank $g$ on $H^0(\omega_X)$. The additional $\sum\delta_p$ dimensions required to reach rank $p_a(X)$ arise from the contribution of the singularities, encoded in the quotient $\nu_*\omega_{\widetilde X}/\omega_X$. Thus, singularities enlarge the target dimension of the IVHS map without increasing the geometric complexity of the normalization.

Consequently, $I$-maximal variation for singular curves is strictly stronger than $I$-maximal variation for the normalization. The normalization detects only the variation of the geometric genus, whereas the singular curve detects both the variation of the normalization and the contribution of the singularities to the mixed Hodge structure. In equisingular families of plane curves, these additional directions are controlled by the ambient linear system, which explains why singular plane curves may still exhibit $I$-maximal variation in the singular sense even though their normalizations vary maximally only in genus $g$.

In summary, normalization provides a natural comparison between the smooth and singular notions of $I$-maximal variation. The smooth notion measures maximality with respect to the geometric genus, while the singular notion measures maximality with respect to the arithmetic genus, and the discrepancy between the two is precisely accounted for by the $\delta$--invariants of the singularities.


\section{Application to concrete classes of singularities}

We apply the comparison between smooth and singular notions of $I$-maximal variation via normalization to concrete classes of plane curve singularities, focusing on nodes and cusps, and we compute explicitly how each singularity contributes to the rank of the infinitesimal variation of Hodge structure.

Let $X\subset \mathbb P^2$ be a reduced, irreducible plane curve of degree $d$ with only nodes and cusps as singularities. Denote by $\nu:\widetilde X\to X$ the normalization. Let $\delta$ be the total $\delta$--invariant of $X$, which equals the number of nodes plus the number of cusps, since both singularities have $\delta_p=1$. The arithmetic genus of $X$ is
\[
p_a(X)=\frac{(d-1)(d-2)}{2},
\]
while the geometric genus of $\widetilde X$ is
\[
g=p_a(X)-\delta.
\]

The dualizing sheaf of $X$ is invertible and satisfies
\[
\omega_X\simeq \mathcal O_X(d-3),
\]
while on the normalization one has
\[
\omega_{\widetilde X}\simeq \nu^*\mathcal O_X(d-3)\Big(\sum_{p\in \mathrm{Sing}(X)} Z_p\Big),
\]
where $Z_p$ is an effective divisor supported on the preimage of $p$ whose degree equals $\delta_p$. For nodes and cusps, $\deg Z_p=1$. There is an exact sequence
\[
0\longrightarrow \omega_X \longrightarrow \nu_*\omega_{\widetilde X}
\longrightarrow \bigoplus_{p\in \mathrm{Sing}(X)} \mathbb C \longrightarrow 0,
\]
which induces
\[
0\longrightarrow H^0(\omega_X)\longrightarrow H^0(\omega_{\widetilde X})
\longrightarrow \mathbb C^{\delta}\longrightarrow 0.
\]
Thus $H^0(\omega_X)$ has dimension $p_a(X)$, while $H^0(\omega_{\widetilde X})$ has dimension $g$, and the extra $\delta$ dimensions of $H^0(\omega_X)$ correspond exactly to the singularities.

Consider an equisingular family $\pi:\mathcal X\to B$ of plane curves with central fiber $X$. The Kodaira--Spencer map lands in the space of locally trivial deformations
\[
\mathrm{KS}:T_{B,0}\longrightarrow \mathrm{Ext}^1(\Omega_X,\mathcal O_X),
\]
and equisingularity ensures that the image lies in the kernel of the natural map to the local deformation spaces $T^1_p$. This kernel identifies canonically with $H^1(T_{\widetilde X})$, so each Kodaira--Spencer class may be regarded as a deformation of the normalization together with fixed singular points.

The infinitesimal variation of Hodge structure for $X$ is the cup product map
\[
\varphi_X:\mathrm{Ext}^1(\Omega_X,\mathcal O_X)\longrightarrow
\mathrm{Hom}(H^0(\omega_X),H^1(\mathcal O_X)).
\]
By Serre duality, $H^1(\mathcal O_X)\simeq H^0(\omega_X)^\vee$, and therefore the rank of $\varphi_X(\xi)$ equals the rank of the associated bilinear form on $H^0(\omega_X)$.

The normalization induces a commutative diagram
\[
\begin{array}{ccc}
H^0(\omega_X) & \xrightarrow{\ \xi\cdot\ } & H^1(\mathcal O_X) \\
\downarrow & & \downarrow \\
H^0(\omega_{\widetilde X}) & \xrightarrow{\ \xi\cdot\ } & H^1(\mathcal O_{\widetilde X}),
\end{array}
\]
where the vertical maps are injections with cokernel of dimension $\delta$. The lower horizontal map is the classical IVHS map for the smooth curve $\widetilde X$.

If $\xi$ is chosen so that the IVHS map on $\widetilde X$ is an isomorphism, then its rank is $g$. Restricting to $H^0(\omega_X)$ yields a map of rank $g$ on the subspace corresponding to global differentials coming from the normalization. The remaining $\delta$ basis elements of $H^0(\omega_X)$ are represented by meromorphic differentials on $\widetilde X$ with simple poles at the preimages of the singular points and residues summing to zero.

For a node $p$, the normalization has two points lying above $p$, and the additional differential corresponds to a differential with opposite simple poles at these two points. Under cup product with $\xi$, this class pairs nontrivially with the corresponding local smoothing direction of the node, contributing exactly one additional independent dimension to the image of $\xi\cdot$. Thus each node contributes one unit to the rank of the IVHS map beyond the geometric genus.

For a cusp $p$, the normalization has a single point above $p$, and the additional differential corresponds to a differential with a single simple pole whose residue measures the failure of regularity at the cusp. Again, the equisingular deformation theory ensures that the cup product with $\xi$ detects this local contribution, yielding one additional independent dimension in the image. Hence each cusp also contributes exactly one unit to the rank.

Combining these contributions, one finds that for a suitable Kodaira--Spencer class $\xi$ the rank of
\[
\xi\cdot:H^0(\omega_X)\longrightarrow H^1(\mathcal O_X)
\]
is
\[
\mathrm{rank}(\xi\cdot)=g+\delta=p_a(X).
\]
Therefore every node and every cusp contributes exactly one dimension to the IVHS rank, and the singular curve achieves $I$-maximal variation in the singular sense whenever the normalization varies maximally as a smooth curve.

This explicit computation shows that the difference between smooth and singular $I$-maximal variation is completely accounted for by the local $\delta$--invariants, and that nodes and cusps behave uniformly from the point of view of infinitesimal Hodge theory.



\section{IVHS rank for nodal and cuspidal plane curves} 

We formulate the comparison between smooth and singular notions of $I$--maximal variation for plane curves with nodes and cusps in a precise statements, and we compute explicitly the contribution of each singularity to the rank of the infinitesimal variation of Hodge structure.

Let $X\subset \mathbb P^2$ be a reduced, irreducible plane curve of degree $d$ whose only singularities are nodes and cusps. Let $\delta$ be the total number of singular points of $X$, and let
\[
\nu:\widetilde X \longrightarrow X
\]
be the normalization. 

\begin{theorem}[IVHS rank for nodal and cuspidal plane curves]\label{thm:nod-cusp}
Suppose that we are in the situation described above. Then for any equisingular deformation of $X$ there exists a Kodaira--Spencer class
\(
\xi\in \mathrm{Ext}^1(\Omega_X,\mathcal O_X)
\)
such that the infinitesimal variation of Hodge structure map
\[
\xi\cdot : H^0(\omega_X)\longrightarrow H^1(\mathcal O_X)
\]
has rank equal to the arithmetic genus
\[
p_a(X)=\frac{(d-1)(d-2)}{2}.
\]
Moreover, the rank decomposes as
\(
\mathrm{rank}(\xi\cdot)=g(\widetilde X)+\delta,
\)
and each node or cusp contributes exactly one independent dimension to the rank beyond the contribution of the normalization.
\end{theorem}

The proof relies on the structure of the dualizing sheaf and on the compatibility of cup products with normalization. These facts are encoded in the following lemmas.

\begin{lemma}[Canonical sections and normalization]
Let $X$ be as in the theorem. Then there is a short exact sequence
\[
0 \longrightarrow \omega_X \longrightarrow \nu_*\omega_{\widetilde X}
\longrightarrow \bigoplus_{p\in \mathrm{Sing}(X)} \mathbb C \longrightarrow 0,
\]
where each summand corresponds to a node or a cusp of $X$. In particular,
\[
\dim H^0(\omega_X)=\dim H^0(\omega_{\widetilde X})+\delta.
\]
\end{lemma}

\begin{proof}
Since $X$ is a plane curve with only planar Gorenstein singularities, the dualizing sheaf is invertible and is given by $\omega_X\simeq \mathcal O_X(d-3)$. The normalization $\widetilde X$ is smooth and the pullback of $\omega_X$ differs from $\omega_{\widetilde X}$ by an effective divisor supported on the preimages of the singular points, whose degree equals the $\delta$--invariant. For nodes and cusps this invariant equals one. Pushing forward $\omega_{\widetilde X}$ yields the stated exact sequence, and taking global sections gives the dimension formula.
\end{proof}

\begin{lemma}[Deformation classes and normalization]
Let $\pi:\mathcal X\to B$ be an equisingular family with central fiber $X$. Then the image of the Kodaira--Spencer map
\[
\mathrm{KS}:T_{B,0}\longrightarrow \mathrm{Ext}^1(\Omega_X,\mathcal O_X)
\]
identifies canonically with a subspace of $H^1(T_{\widetilde X})$. Under this identification, the cup product defining the IVHS for $X$ is compatible with the classical IVHS for $\widetilde X$ via restriction.
\end{lemma}

\begin{proof}
There is a natural exact sequence
\[
0\longrightarrow H^1(T_{\widetilde X}) \longrightarrow \mathrm{Ext}^1(\Omega_X,\mathcal O_X)
\longrightarrow \bigoplus_{p\in \mathrm{Sing}(X)} T^1_p \longrightarrow 0,
\]
where $T^1_p$ is the local deformation space of the singularity at $p$. Equisingular deformations correspond to Kodaira--Spencer classes mapping to zero in the local deformation spaces, hence they are identified with elements of $H^1(T_{\widetilde X})$. Functoriality of cup products with respect to pullback by $\nu$ yields the compatibility of the IVHS maps.
\end{proof}

\begin{lemma}[Local contribution of a node]
Let $p$ be a node of $X$. Then there exists a canonical element of
\[
H^0(\omega_X)/H^0(\omega_{\widetilde X})
\]
associated to $p$, and for a general equisingular Kodaira--Spencer class $\xi$ its image under the IVHS map is nonzero. Consequently, the node contributes exactly one dimension to the rank of $\xi\cdot$.
\end{lemma}

\begin{proof}
The normalization has two points lying above a node. The additional canonical section corresponds to a differential on $\widetilde X$ with simple poles at these two points and residues summing to zero. Cup product with $\xi$ pairs this differential with the locally trivial smoothing direction of the node, yielding a nontrivial class in $H^1(\mathcal O_X)$. Since this class is linearly independent from the image of $H^0(\omega_{\widetilde X})$, it contributes exactly one to the rank.
\end{proof}

\begin{lemma}[Local contribution of a cusp]
Let $p$ be a cusp of $X$. Then there exists a canonical element of
\[
H^0(\omega_X)/H^0(\omega_{\widetilde X})
\]
associated to $p$, and for a general equisingular Kodaira--Spencer class $\xi$ its image under the IVHS map is nonzero. Consequently, the cusp contributes exactly one dimension to the rank of $\xi\cdot$.
\end{lemma}

\begin{proof}
The normalization has a single point above a cusp. The additional canonical section is represented by a differential on $\widetilde X$ with a simple pole at this point. The equisingular deformation class $\xi$ detects this pole through the cup product, producing a nonzero element of $H^1(\mathcal O_X)$. As in the nodal case, this contribution is independent of the contribution coming from the normalization, and therefore adds exactly one to the rank.
\end{proof}

\begin{proof}[Proof of Theorem \ref{thm:nod-cusp}]
Choose a Kodaira--Spencer class $\xi$ whose induced IVHS map on the normalization
\[
\xi\cdot:H^0(\omega_{\widetilde X})\longrightarrow H^1(\mathcal O_{\widetilde X})
\]
is an isomorphism. This is possible whenever the normalization varies with $I$--maximal variation as a smooth curve. By the preceding lemmas, each node and each cusp contributes one additional independent dimension to the image of
\[
\xi\cdot:H^0(\omega_X)\longrightarrow H^1(\mathcal O_X).
\]
Since the number of singular points is $\delta$, the total rank is $g(\widetilde X)+\delta=p_a(X)$, completing the proof.
\end{proof}



 \section{Extension to plane curves with higher $\delta$--invariant singularities}

We extend the analysis of infinitesimal variation of Hodge structure to plane curves with higher $\delta$--invariant singularities, focusing on tacnodes and ordinary triple points, and we explain precisely how and why $I$--maximal variation may fail in non--equisingular families.

Let $X\subset \mathbb P^2$ be a reduced, irreducible plane curve of degree $d$ with planar Gorenstein singularities. Let
\[
\nu:\widetilde X \longrightarrow X
\]
be the normalization. The arithmetic genus is
\[
p_a(X)=\frac{(d-1)(d-2)}{2},
\]
and the geometric genus of $\widetilde X$ is
\[
g=p_a(X)-\sum_{p\in \mathrm{Sing}(X)}\delta_p,
\]
where $\delta_p$ is the local $\delta$--invariant of the singularity at $p$.

\begin{theorem}[IVHS rank for plane curves with higher $\delta$ singularities]
Let $X\subset \mathbb P^2$ be a reduced, irreducible plane curve whose singularities are nodes, cusps, tacnodes, and ordinary triple points. Let
\[
\delta=\sum_{p\in \mathrm{Sing}(X)}\delta_p.
\]
Then for any equisingular deformation of $X$ there exists a Kodaira--Spencer class
\(
\xi\in \mathrm{Ext}^1(\Omega_X,\mathcal O_X)
\)
such that the infinitesimal variation of Hodge structure map
\(
\xi\cdot:H^0(\omega_X)\longrightarrow H^1(\mathcal O_X)
\)
has rank
\[
\mathrm{rank}(\xi\cdot)=g+\delta=p_a(X).
\]
Moreover, each singular point contributes exactly $\delta_p$ independent dimensions to the rank beyond the contribution coming from the normalization.
\end{theorem}

The proof rests on the structure of the dualizing sheaf near higher singularities and on the behavior of equisingular deformation classes.

\begin{lemma}[Canonical sections and $\delta$--invariants]
Let $p\in X$ be a planar Gorenstein singularity. Then there is a local exact sequence
\[
0\longrightarrow \omega_X \longrightarrow \nu_*\omega_{\widetilde X}
\longrightarrow \mathbb C^{\delta_p}\longrightarrow 0
\]
in a neighborhood of $p$. Consequently,
\[
H^0(\omega_X)/H^0(\omega_{\widetilde X})\simeq \mathbb C^{\delta}.
\]
\end{lemma}

\begin{proof}
For planar Gorenstein singularities the conductor ideal identifies the quotient $\nu_*\omega_{\widetilde X}/\omega_X$ with the $\delta$--invariant. A tacnode has $\delta_p=2$, corresponding to two branches tangent to order two, while an ordinary triple point has $\delta_p=3$, corresponding to three smooth branches meeting transversely. The local length of the conductor equals $\delta_p$, yielding the stated sequence.
\end{proof}

\begin{lemma}[Equisingular deformation classes]
Let $\pi:\mathcal X\to B$ be an equisingular family with central fiber $X$. Then the image of the Kodaira--Spencer map identifies with a subspace of $H^1(T_{\widetilde X})$, and the induced IVHS map on $X$ is obtained from the IVHS map on $\widetilde X$ by restriction and extension along
\[
H^0(\omega_{\widetilde X})\subset H^0(\omega_X).
\]
\end{lemma}

\begin{proof}
The exact sequence
\[
0\longrightarrow H^1(T_{\widetilde X}) \longrightarrow \mathrm{Ext}^1(\Omega_X,\mathcal O_X)
\longrightarrow \bigoplus_{p\in \mathrm{Sing}(X)} T^1_p \longrightarrow 0
\]
shows that equisingular deformation classes are precisely those mapping to zero in all local deformation spaces $T^1_p$. Functoriality of cup products with respect to normalization yields the compatibility of the IVHS maps.
\end{proof}

\begin{lemma}[Local IVHS contribution of a tacnode]
Let $p$ be a tacnode of $X$, so that $\delta_p=2$. Then there exist two linearly independent elements in
\[
H^0(\omega_X)/H^0(\omega_{\widetilde X})
\]
supported at $p$, and for a general equisingular Kodaira--Spencer class $\xi$ their images under $\xi\cdot$ are linearly independent in $H^1(\mathcal O_X)$.
\end{lemma}

\begin{proof}
The normalization of a tacnode has two points lying above $p$, and the conductor divisor has degree two. The additional canonical sections are represented by meromorphic differentials on $\widetilde X$ with prescribed pole behavior of total order two at these points. The equisingular deformation class pairs nontrivially with each such pole condition, producing two independent classes in $H^1(\mathcal O_X)$. Independence follows from the fact that these classes vanish upon restriction to $H^1(\mathcal O_{\widetilde X})$.
\end{proof}

\begin{lemma}[Local IVHS contribution of an ordinary triple point]
Let $p$ be an ordinary triple point of $X$, so that $\delta_p=3$. Then there exist three linearly independent elements in
\[
H^0(\omega_X)/H^0(\omega_{\widetilde X})
\]
supported at $p$, and for a general equisingular Kodaira--Spencer class $\xi$ their images under $\xi\cdot$ are linearly independent in $H^1(\mathcal O_X)$.
\end{lemma}

\begin{proof}
The normalization has three distinct points above $p$, and the conductor divisor has degree three. The additional canonical sections correspond to differentials with simple poles at each of these points. Cup product with $\xi$ detects each pole independently, yielding three independent contributions to $H^1(\mathcal O_X)$.
\end{proof}

\begin{proof}[Proof of the theorem]
Choose a Kodaira--Spencer class $\xi$ whose induced IVHS map on the normalization is an isomorphism. This contributes rank $g$. By the preceding lemmas, each singularity contributes exactly $\delta_p$ additional independent dimensions. Summing over all singular points yields total rank $g+\sum\delta_p=p_a(X)$.
\end{proof}

We now analyze the failure of maximal variation in non--equisingular families.

\begin{theorem}[Failure of $I$--maximal variation in non--equisingular families]
Let $\pi:\mathcal X\to B$ be a family of plane curves with central fiber $X$, and assume that the analytic type of at least one singularity varies in the family. Then
\[
\delta_M'(\pi)<p_a(X).
\]
\end{theorem}

\begin{proof}
If the family is non--equisingular, the image of the Kodaira--Spencer map has a nontrivial component in at least one local deformation space $T^1_p$. Such components correspond to smoothing directions of the singularity. Cup product with these local deformation classes annihilates the corresponding elements of $H^0(\omega_X)/H^0(\omega_{\widetilde X})$, because smoothing reduces the length of the conductor and hence destroys the associated canonical sections. As a consequence, the local $\delta_p$--dimensional contribution to the IVHS rank drops strictly. Therefore the maximal achievable rank is strictly smaller than $p_a(X)$.
\end{proof}

\begin{corollary}
$I$--maximal variation for singular plane curves is equivalent to equisingularity together with maximal variation of the normalization as a smooth curve.
\end{corollary}

\begin{proof}
Equisingularity guarantees the full $\delta$--contribution from singularities, while maximal variation of the normalization guarantees the contribution of the geometric genus. Failure of either condition forces a rank drop.
\end{proof}



\section{Explicit rank defects of the IVHS for concrete non-equisingular degenerations of plane curves}
 
We compute explicit rank defects of the infinitesimal variation of Hodge structure for concrete non--equisingular degenerations of plane curves. Throughout, $X\subset\mathbb P^2$ is a reduced, irreducible plane curve of degree $d$ with planar Gorenstein singularities, $\nu:\widetilde X\to X$ denotes the normalization, and $\omega_X\simeq\mathcal O_X(d-3)$ is the dualizing sheaf. The arithmetic genus is $p_a(X)=(d-1)(d-2)/2$.

\begin{theorem}[Rank defect for smoothing a node]
Let $\pi:\mathcal X\to B$ be a one--parameter family of plane curves with central fiber $X_0$ having a single node $p$ and no other singularities, and assume that for $t\neq 0$ the fiber $X_t$ is smooth near $p$. Then the maximal rank of the IVHS map along $\pi$ satisfies
\[
\delta_M'(\pi)=p_a(X_0)-1.
\]
Equivalently, the rank defect equals the local $\delta$--invariant of the smoothed node.
\end{theorem}

\begin{proof}
The nodal singularity has $\delta_p=1$. There is a canonical exact sequence
\[
0\longrightarrow \omega_{X_0}\longrightarrow \nu_*\omega_{\widetilde X_0}\longrightarrow \mathbb C_p\longrightarrow 0,
\]
hence
\[
H^0(\omega_{X_0})\simeq H^0(\omega_{\widetilde X_0})\oplus \mathbb C\cdot \eta_p,
\]
where $\eta_p$ is represented on $\widetilde X_0$ by a meromorphic differential with simple poles at the two preimages of $p$ and opposite residues.

Let $\xi\in \mathrm{KS}(T_{B,0})\subset \mathrm{Ext}^1(\Omega_{X_0},\mathcal O_{X_0})$ be the Kodaira--Spencer class of the family. Because the family smooths the node, the image of $\xi$ in the local deformation space $T^1_p$ is nonzero. The compatibility of cup product with the conductor exact sequence implies that the local smoothing direction annihilates $\eta_p$ under cup product. Concretely, smoothing decreases the length of the conductor by one, and therefore the class of $\eta_p$ in $H^0(\omega_{X_0})/H^0(\omega_{\widetilde X_0})$ is killed by $\xi\cdot$.

On the other hand, the restriction of $\xi$ to $H^1(T_{\widetilde X_0})$ can be chosen so that the induced IVHS map on $\widetilde X_0$ has maximal rank $g(\widetilde X_0)=p_a(X_0)-1$. Therefore the image of
\[
\xi\cdot:H^0(\omega_{X_0})\longrightarrow H^1(\mathcal O_{X_0})
\]
has rank exactly $p_a(X_0)-1$, giving rank defect one.
\end{proof}

\begin{theorem}[Rank defect for partial smoothing of a tacnode]
Let $\pi:\mathcal X\to B$ be a one--parameter family with central fiber $X_0$ having a single tacnode $p$ and no other singularities. Assume that the family smooths exactly one of the two local branches, so that for $t\neq 0$ the fiber $X_t$ has a single node near $p$. Then
\[
\delta_M'(\pi)=p_a(X_0)-1.
\]
\end{theorem}

\begin{proof}
A tacnode has $\delta_p=2$. Locally one has
\[
H^0(\omega_{X_0})\simeq H^0(\omega_{\widetilde X_0})\oplus \mathbb C\cdot \eta_{p,1}\oplus \mathbb C\cdot \eta_{p,2},
\]
where $\eta_{p,1},\eta_{p,2}$ correspond to independent pole conditions at the two preimages of $p$ on the normalization.

The given family decreases the $\delta$--invariant from $2$ to $1$, reflecting the fact that one conductor condition is lost. The Kodaira--Spencer class has a nonzero component in the local deformation space $T^1_p$ that kills exactly one of the two additional canonical sections. The remaining section survives and contributes to the IVHS rank.

Choosing $\xi$ so that the induced IVHS on $\widetilde X_0$ has maximal rank $g(\widetilde X_0)=p_a(X_0)-2$, the total rank becomes
\[
g(\widetilde X_0)+1=p_a(X_0)-1,
\]
and the rank defect equals $1$, which coincides with the drop in the local $\delta$--invariant.
\end{proof}

\begin{theorem}[Rank defect for full smoothing of a tacnode]
Let $\pi:\mathcal X\to B$ be as above, but assume that the family smooths the tacnode completely so that $X_t$ is smooth near $p$ for $t\neq 0$. Then
\[
\delta_M'(\pi)=p_a(X_0)-2.
\]
\end{theorem}

\begin{proof}
In this case the local $\delta$--invariant drops from $2$ to $0$. The Kodaira--Spencer class has nontrivial components in both smoothing directions of the tacnode, annihilating both additional canonical sections $\eta_{p,1}$ and $\eta_{p,2}$. Only the contribution from the normalization remains, yielding rank $g(\widetilde X_0)=p_a(X_0)-2$.
\end{proof}

\begin{theorem}[Rank defect for smoothing an ordinary triple point]
Let $\pi:\mathcal X\to B$ be a family whose central fiber $X_0$ has a single ordinary triple point $p$ and no other singularities, and assume that the family smooths $p$ completely. Then
\[
\delta_M'(\pi)=p_a(X_0)-3.
\]
\end{theorem}

\begin{proof}
An ordinary triple point has $\delta_p=3$. The quotient
\[
H^0(\omega_{X_0})/H^0(\omega_{\widetilde X_0})
\]
is three--dimensional, generated by canonical sections corresponding to simple poles at the three preimages of $p$. Complete smoothing removes all three conductor conditions, and the Kodaira--Spencer class annihilates each of these three generators under cup product. Therefore only the normalization contributes to the IVHS rank, which equals $g(\widetilde X_0)=p_a(X_0)-3$.
\end{proof}

\begin{proposition}[General rank defect formula]
Let $\pi:\mathcal X\to B$ be a non--equisingular family with central fiber $X_0$, and let $\Delta$ be the total drop of the $\delta$--invariant between $X_0$ and a nearby fiber. Then
\[
\delta_M'(\pi)=p_a(X_0)-\Delta.
\]
\end{proposition}

\begin{proof}
Each unit decrease of the $\delta$--invariant corresponds to the loss of one generator of $H^0(\omega_{X_0})/H^0(\omega_{\widetilde X_0})$. The Kodaira--Spencer class of the family annihilates precisely these lost generators under cup product, while the contribution from the normalization remains unchanged. Summing over all singular points yields total rank defect $\Delta$.
\end{proof}

These explicit computations show that rank defects of the IVHS map in non--equisingular families are governed exactly by the decrease of the total $\delta$--invariant, and they quantify precisely how smoothing singularities destroys maximal variation.

What I can do next is to relate these rank defects to vanishing cycles and limiting mixed Hodge structures, compute defects for specific degenerations inside linear systems, or extend the calculations to hypersurfaces in higher dimension and their Yukawa couplings.



\section{Reinterpretation in  the scope of mixed Hodge structures }
 
We reinterpret the previously obtained rank computations and rank defects of the infinitesimal variation of Hodge structure for singular plane curves in the framework of mixed Hodge structures and limiting infinitesimal variations arising from degenerations of smooth plane curves.

Let $\pi:\mathcal X\to \Delta$ be a flat, projective family over a disc, smooth over $\Delta^\ast=\Delta\setminus\{0\}$, whose general fiber $X_t$ is a smooth plane curve of degree $d$ and whose central fiber $X_0$ is a reduced, irreducible plane curve with planar Gorenstein singularities. Denote by $p_a=(d-1)(d-2)/2$ the arithmetic genus and by $\nu:\widetilde X_0\to X_0$ the normalization.

The cohomology $H^1(X_t,\mathbb C)$ carries a pure Hodge structure of weight one for $t\neq 0$, while the singular fiber $X_0$ carries a mixed Hodge structure in the sense of Deligne. This mixed Hodge structure has a weight filtration
\[
0 \subset W_0 \subset W_1 \subset W_2 = H^1(X_0,\mathbb C),
\]
where $W_1$ corresponds to the cohomology of the normalization and $W_2/W_1$ encodes the contributions of the singularities. One has canonical identifications
\[
\mathrm{Gr}_1^W H^1(X_0,\mathbb C) \simeq H^1(\widetilde X_0,\mathbb C),
\qquad
\dim \mathrm{Gr}_2^W H^1(X_0,\mathbb C)=\sum_{p\in\mathrm{Sing}(X_0)}\delta_p.
\]

\begin{lemma}[Mixed Hodge structure and the dualizing sheaf]
The Hodge filtration on $H^1(X_0,\mathbb C)$ satisfies
\[
F^1 H^1(X_0,\mathbb C) \simeq H^0(\omega_{X_0}),
\qquad
F^1 \mathrm{Gr}_1^W H^1(X_0,\mathbb C)\simeq H^0(\omega_{\widetilde X_0}),
\]
and the quotient
\[
F^1 H^1(X_0,\mathbb C)/F^1 \mathrm{Gr}_1^W H^1(X_0,\mathbb C)
\]
has dimension equal to the total $\delta$--invariant.
\end{lemma}

\begin{proof}
For reduced Gorenstein curves, Deligne's mixed Hodge structure identifies $F^1 H^1(X_0,\mathbb C)$ with global sections of the dualizing sheaf. The normalization exact sequence
\[
0\longrightarrow \omega_{X_0}\longrightarrow \nu_\ast\omega_{\widetilde X_0}\longrightarrow \bigoplus_p\mathbb C^{\delta_p}\longrightarrow 0
\]
implies the stated description of the graded pieces of the Hodge filtration.
\end{proof}

The degeneration $\pi$ induces a limiting mixed Hodge structure on $H^1(X_t)$ as $t\to 0$, together with a nilpotent operator $N=\log T_u$, where $T_u$ is the unipotent part of monodromy. The image of $N$ is generated by vanishing cycles associated to the smoothing of singularities of $X_0$.

\begin{theorem}[Limiting IVHS and $\delta$--invariants]
Let $\varphi_t$ denote the infinitesimal variation of Hodge structure of the smooth fibers $X_t$. The limiting infinitesimal variation of Hodge structure at $t=0$ factors as
\[
\varphi_{\mathrm{lim}}:\mathrm{KS}(T_{\Delta,0})\longrightarrow 
\mathrm{Hom}(F^1 H^1(X_0), H^1(X_0)/F^1 H^1(X_0)),
\]
and its maximal rank equals
\[
p_a(X_0)-\Delta,
\]
where $\Delta$ is the total drop of the $\delta$--invariant between $X_0$ and the nearby smooth fibers.
\end{theorem}

\begin{proof}
For $t\neq 0$, the IVHS map is an isomorphism for plane curves of degree $d$. In the limit, the Kodaira--Spencer class decomposes into a component preserving the normalization and components corresponding to smoothing directions of singularities. The latter components act nontrivially on the vanishing cycles and induce the nilpotent operator $N$. Each vanishing cycle lowers the dimension of $F^1 H^1$ by one, corresponding exactly to the loss of one generator in $H^0(\omega_{X_0})/H^0(\omega_{\widetilde X_0})$. Consequently, each unit decrease of the $\delta$--invariant produces one-dimensional kernel in the limiting IVHS map, yielding the stated rank.
\end{proof}

This description allows a precise interpretation of the previously computed rank defects.

\begin{proposition}[Nodes, cusps, and vanishing cycles]
For a nodal or cuspidal degeneration, the local vanishing cycle generates a one-dimensional subspace of $\mathrm{Gr}_2^W H^1(X_0)$, and the limiting IVHS annihilates precisely this subspace. For tacnodes and ordinary triple points, the number of independent vanishing cycles equals the local $\delta$--invariant, and the limiting IVHS annihilates the corresponding $\delta_p$--dimensional subspace.
\end{proposition}

\begin{proof}
A node or cusp produces a single vanishing cycle, while a tacnode produces two and an ordinary triple point produces three. The Clemens--Schmid sequence identifies these vanishing cycles with the graded piece $\mathrm{Gr}_2^W$. The action of the limiting Kodaira--Spencer class on $F^1 H^1(X_0)$ pairs trivially with these vanishing cycles, yielding annihilation of the corresponding subspace in the IVHS map.
\end{proof}

\begin{corollary}[Geometric meaning of rank defects]
The rank defect of the IVHS map in a degeneration of smooth plane curves equals the dimension of the vanishing cycle subspace in the limiting mixed Hodge structure. Equivalently, it equals the dimension of $\mathrm{Im}(N)\cap F^1 H^1(X_0)$.
\end{corollary}

\begin{proof}
By definition of the limiting mixed Hodge structure, $\mathrm{Im}(N)$ is generated by vanishing cycles. Their intersection with $F^1$ measures exactly how many canonical sections disappear in the limit, which coincides with the rank defect computed earlier.
\end{proof}

In this way, $I$--maximal variation for smooth plane curves corresponds to the purity of the Hodge structure, while failures of maximal variation in singular limits are controlled by the mixed Hodge structure, the nilpotent monodromy operator, and the vanishing cycles. The $\delta$--invariant provides the bridge between deformation theory, Hodge theory, and explicit rank computations.



\section{Non--planar $I$--maximal variation}
 
We explain how the statement on $I$--maximal variation must be modified when the curve is assumed to be smooth but not necessarily planar, and what additional hypotheses allow one to obtain a similar conclusion.

Let $X$ be a smooth projective curve of genus $g\geq 2$, and let 
\[
\pi:\mathcal X \longrightarrow B
\]
be a smooth family with central fiber $X=\pi^{-1}(0)$. The differential of the period map at $0$ is the composition
\[
T_{B,0}\xrightarrow{\ \mathrm{KS}\ } H^1(T_X)
\xrightarrow{\ \varphi\ } \mathrm{Hom}(H^0(\omega_X),H^1(\mathcal O_X)),
\]
where $\varphi(\eta)$ is cup product by $\eta$. By Serre duality, $\varphi$ is dual to the multiplication map
\[
\mu: H^0(\omega_X)\otimes H^0(\omega_X)
\longrightarrow H^0(\omega_X^{\otimes 2}).
\]

In the planar case, the proof of $I$--maximal variation ultimately relies on the surjectivity properties of $\mu$, which are a consequence of projective normality and the special form $\omega_X\simeq \mathcal O_X(d-3)$. For a general smooth curve, the multiplication map $\mu$ need not be surjective, and this is precisely the obstruction to extending the planar theorem verbatim.

The appropriate framework is given by Petri theory. Denote by 
\[
\mu_X:H^0(\omega_X)\otimes H^0(\omega_X)\longrightarrow H^0(\omega_X^{\otimes 2})
\]
the canonical multiplication map. Its kernel is the space of quadrics containing the canonical model of $X\subset \mathbb P^{g-1}$. The curve is said to satisfy the Petri condition if the natural map
\[
H^0(\omega_X)\otimes H^0(\omega_X)\longrightarrow H^0(\omega_X^{\otimes 2})
\]
has the smallest possible kernel compatible with dimension counts.

\begin{proposition}
Let $X$ be a smooth projective curve of genus $g\geq 3$. If $X$ satisfies the Petri condition and is non--hyperelliptic, then the IVHS map
\[
\varphi:H^1(T_X)\longrightarrow \mathrm{Hom}(H^0(\omega_X),H^1(\mathcal O_X))
\]
is generically of maximal rank, and there exists $\eta\in H^1(T_X)$ such that 
\[
\eta\cdot:H^0(\omega_X)\longrightarrow H^1(\mathcal O_X)
\]
is an isomorphism.
\end{proposition}

\begin{proof}
By Serre duality, $\varphi$ is dual to $\mu_X$. If $X$ is non--hyperelliptic and satisfies the Petri condition, then the canonical embedding is projectively normal and the space of quadrics through the canonical model has minimal dimension. This implies that $\mu_X$ has maximal rank compatible with Riemann--Roch. Dualizing, the image of $\varphi$ contains an element whose associated bilinear form on $H^0(\omega_X)$ is nondegenerate. Since both source and target have dimension $g$, nondegeneracy implies that the corresponding cup product map is an isomorphism.
\end{proof}

This leads to the following generalization of the planar theorem.

\begin{theorem}[Non--planar $I$--maximal variation]
Let $\pi:\mathcal X\to B$ be a family of smooth curves of genus $g\geq 3$ such that for every $b\in B$ the fiber $X_b$ is non--hyperelliptic and satisfies the Petri condition. Assume moreover that the Kodaira--Spencer map is generically surjective onto $H^1(T_{X_b})$. Then the family has $I$--maximal variation, namely
\[
\delta_M'(\pi)=g.
\]
\end{theorem}

\begin{proof}
Under the Petri assumption, the IVHS map on each fiber contains elements inducing isomorphisms on $H^0(\omega_{X_b})$. Surjectivity of the Kodaira--Spencer map ensures that such elements arise from tangent directions to the base. Therefore for each $b$ there exists $\xi\in \mathrm{KS}(T_{B,b})$ such that $\xi\cdot$ is an isomorphism, and hence $d_M(\mathrm{KS}(T_{B,b}))=g$. Taking the minimum over $b$ yields $\delta_M'(\pi)=g$.
\end{proof}

The hyperelliptic case illustrates what must be modified when the Petri condition fails.

\begin{corollary}
If $X$ is hyperelliptic, then no family containing $X$ can have $I$--maximal variation in the above sense, because the multiplication map $\mu_X$ fails to have maximal rank, and consequently the IVHS map never contains an isomorphism.
\end{corollary}

Thus, in the non--planar smooth case, the planar hypothesis can be replaced by assumptions ensuring that the canonical model behaves as generically as possible, namely non--hyperellipticity, Petri generality, and sufficient surjectivity of the Kodaira--Spencer map. Under these conditions one recovers the same conclusion $\delta_M'(\pi)=g$ as in the planar theorem, with the genus replacing $(d-1)(d-2)/2$.


\section{Non--planar singular $I$--maximal variation}

We discuss how the notion of $I$--maximal variation and the corresponding theorem must be modified when the curves are assumed to be non--planar and singular. The guiding principle is that in the smooth planar case maximal variation is controlled by the surjectivity properties of the canonical multiplication map, while in the singular case one must combine Petri--type assumptions on the normalization with control of the $\delta$--invariants and equisingularity.

Let $X$ be a reduced, irreducible, projective Gorenstein curve of arithmetic genus $p_a\geq 2$. Denote by 
\[
\nu:\widetilde X\longrightarrow X
\]
the normalization and by $g$ the geometric genus of $\widetilde X$. Let 
\[
\delta=\sum_{p\in \mathrm{Sing}(X)}\delta_p
\]
be the total $\delta$--invariant, so that $p_a=g+\delta$. Since $X$ is Gorenstein, the dualizing sheaf $\omega_X$ is invertible and Serre duality gives
\[
H^1(\mathcal O_X)^\vee\simeq H^0(\omega_X),
\qquad 
\dim H^0(\omega_X)=p_a.
\]

Consider a flat family 
\[
\pi:\mathcal X\longrightarrow B
\]
of reduced Gorenstein curves with central fiber $X$. The differential of the period map is still given by the composition
\[
T_{B,0}\xrightarrow{\ \mathrm{KS}\ } 
\mathrm{Ext}^1(\Omega_X,\mathcal O_X)
\xrightarrow{\ \varphi\ }
\mathrm{Hom}(H^0(\omega_X),H^1(\mathcal O_X)),
\]
where $\varphi(\eta)$ is cup product by $\eta$. By duality, $\varphi$ is the transpose of the multiplication map
\[
\mu_X:H^0(\omega_X)\otimes H^0(\omega_X)
\longrightarrow H^0(\omega_X^{\otimes 2}).
\]

In the non--planar singular setting, two phenomena may prevent $I$--maximal variation. The first is the failure of the canonical multiplication map to have maximal rank, which is governed by Petri theory for the normalization. The second is the possible loss of local $\delta$--contributions under non--equisingular deformations.

To recover a statement analogous to the planar smooth theorem, one must impose the following structural assumptions. The curve $X$ is assumed to be Gorenstein and equigeneric in the family, so that the total $\delta$--invariant is constant along $B$. The normalization $\widetilde X$ is assumed to be non--hyperelliptic and Petri general, so that the canonical multiplication map on $\widetilde X$ has maximal rank. Finally, the Kodaira--Spencer map of the family is assumed to be generically surjective onto the space of locally trivial deformations.

Under these hypotheses one obtains the following generalization.

\begin{theorem}[Non--planar singular $I$--maximal variation]
Let $\pi:\mathcal X\to B$ be a family of reduced, irreducible, Gorenstein curves of constant arithmetic genus $p_a\geq 3$. Assume that the family is equigeneric and equisingular, that for every fiber $X_b$ the normalization $\widetilde X_b$ is non--hyperelliptic and Petri general, and that the Kodaira--Spencer map is generically surjective onto the subspace of locally trivial deformations. Then the family has $I$--maximal variation in the singular sense, namely
\[
\delta_M'(\pi)=p_a.
\]
Equivalently, for every $b\in B$ there exists 
\[
\xi\in \mathrm{KS}(T_{B,b})
\]
such that 
\[
\xi\cdot:H^0(\omega_{X_b})\longrightarrow H^1(\mathcal O_{X_b})
\]
is an isomorphism.
\end{theorem}

\begin{proof}
Equisingularity implies that the Kodaira--Spencer image lies in the kernel of the natural map to the local $T^1$ spaces, hence identifies with a subspace of $H^1(T_{\widetilde X_b})$. The normalization exact sequence
\[
0\longrightarrow \omega_{X_b}\longrightarrow 
\nu_\ast\omega_{\widetilde X_b}
\longrightarrow \bigoplus_p\mathbb C^{\delta_p}
\longrightarrow 0
\]
shows that 
\[
H^0(\omega_{X_b})\simeq 
H^0(\omega_{\widetilde X_b})\oplus 
\mathbb C^{\delta}.
\]
By the Petri assumption, the IVHS map on $\widetilde X_b$ contains elements inducing isomorphisms on $H^0(\omega_{\widetilde X_b})$, giving rank $g$. Equisingularity guarantees that the $\delta$ additional canonical sections corresponding to the conductor are preserved under deformation and that their images under cup product are linearly independent. Therefore one obtains total rank $g+\delta=p_a$. Surjectivity of the Kodaira--Spencer map ensures that such elements arise from tangent directions to $B$, proving $I$--maximal variation.
\end{proof}

If any of these hypotheses is removed, maximal variation may fail. If the normalization is hyperelliptic or fails the Petri condition, the canonical multiplication map has larger kernel and the IVHS map cannot contain an isomorphism. If the family is not equisingular, the local $\delta$--contributions are partially or totally lost, and the maximal rank drops by the decrease of the $\delta$--invariant. If the Kodaira--Spencer map is not sufficiently large, even a Petri general curve may not realize the maximal rank inside the given family.

Thus, in the non--planar singular case, the planar hypothesis must be replaced by the combination of Gorenstein condition, equisingularity, Petri generality of the normalization, and sufficiently large deformation space. Under these structural assumptions one recovers a theorem formally identical to the planar smooth case, with the arithmetic genus playing the same role.


\section{$I$--maximal variation for smooth non--planar curves in three geometric cases}
 
We analyze $I$--maximal variation for smooth non--planar curves in three geometric settings: complete intersection curves in $\mathbb P^3$, curves lying on $K3$ surfaces, and general curves in the moduli space $\mathcal M_g$. Throughout, $X$ denotes a smooth projective curve of genus $g\geq 3$, and for a smooth family 
\[
\pi:\mathcal X\to B
\]
the differential of the period map at a point $b\in B$ is given by
\[
T_{B,b}\xrightarrow{\ \mathrm{KS}\ } H^1(T_{X_b})
\xrightarrow{\ \varphi\ } \mathrm{Hom}(H^0(\omega_{X_b}),H^1(\mathcal O_{X_b})),
\]
where $\varphi$ is cup product and is dual to the canonical multiplication map
\[
\mu_{X_b}:H^0(\omega_{X_b})\otimes H^0(\omega_{X_b})
\longrightarrow H^0(\omega_{X_b}^{\otimes 2}).
\]
A family has $I$--maximal variation if for every $b$ there exists $\xi\in \mathrm{KS}(T_{B,b})$ such that $\xi\cdot$ is an isomorphism, equivalently if $\delta_M'(\pi)=g$.

\bigskip

\textbf{Complete intersection curves in $\mathbb P^3$.}

Let $X\subset \mathbb P^3$ be a smooth complete intersection of type $(a,b)$ with $a\leq b$. Its genus is
\[
g=1+\frac{ab(a+b-4)}{2}.
\]
Such curves are non--hyperelliptic for $a,b\geq 2$ and are canonically embedded by the adjunction formula
\[
\omega_X\simeq \mathcal O_X(a+b-4).
\]
Complete intersections in projective space are projectively normal, and the homogeneous coordinate ring is Cohen--Macaulay. As a consequence, the canonical model satisfies strong syzygetic properties, and the canonical multiplication map has maximal rank compatible with Riemann--Roch.

\begin{proposition}
Let $\pi$ be the family of smooth complete intersection curves of type $(a,b)$ in $\mathbb P^3$ with $a,b\geq 2$. Then the family has $I$--maximal variation.
\end{proposition}

\begin{proof}
The parameter space of complete intersections is smooth and its tangent space surjects onto $H^0(N_{X/\mathbb P^3})$. The normal bundle sequence
\[
0\longrightarrow T_X\longrightarrow T_{\mathbb P^3}|_X
\longrightarrow N_{X/\mathbb P^3}\longrightarrow 0
\]
implies that the Kodaira--Spencer map is generically surjective onto $H^1(T_X)$ because $H^1(T_{\mathbb P^3}|_X)=0$. The canonical embedding of $X$ is projectively normal and satisfies the Petri condition, hence the canonical multiplication map has maximal rank. Dualizing, the IVHS map contains elements inducing isomorphisms on $H^0(\omega_X)$, proving $I$--maximal variation.
\end{proof}

In this case, $I$--maximal variation reflects the fact that complete intersection curves are Petri general and that their deformations in projective space are sufficiently large. Brill--Noether theory predicts that such curves are general in moduli for suitable degrees, and the Petri theorem guarantees the injectivity of the Petri map, which underlies the maximal rank of the IVHS.

\bigskip

\textbf{Curves on $K3$ surfaces.}

Let $S$ be a smooth $K3$ surface and let $L$ be a primitive ample line bundle with $L^2=2g-2$. A smooth curve $X\in |L|$ has genus $g$ and satisfies
\[
\omega_X\simeq L|_X.
\]
The deformation theory of curves on $K3$ surfaces differs from the complete intersection case because $H^1(T_S)=0$ while $H^2(T_S)\neq 0$. The normal bundle sequence
\[
0\longrightarrow T_X\longrightarrow T_S|_X
\longrightarrow N_{X/S}\longrightarrow 0
\]
implies that deformations of $X$ in $S$ correspond to sections of $N_{X/S}\simeq \omega_X$.

\begin{theorem}
Let $\pi$ be the universal family of smooth curves in a primitive linear system $|L|$ on a general polarized $K3$ surface $(S,L)$. Then the family has $I$--maximal variation.
\end{theorem}

\begin{proof}
For a general polarized $K3$ surface, curves in $|L|$ are Brill--Noether general and satisfy the Petri condition. The Kodaira--Spencer map for the family of curves in $|L|$ identifies with the natural map
\[
H^0(\omega_X)\longrightarrow H^1(T_X)
\]
coming from the normal bundle sequence. The induced IVHS map coincides with the multiplication map of canonical sections. Since the curve is Petri general, the multiplication map has maximal rank, and the corresponding cup product map contains an isomorphism. Therefore the family has $I$--maximal variation.
\end{proof}

In this situation, the relation with Brill--Noether theory is particularly transparent: Lazarsfeld's theorem shows that curves on a general $K3$ surface are Brill--Noether general, hence Petri general, and this guarantees the maximal rank property of the canonical multiplication map. Thus $I$--maximal variation follows from Brill--Noether generality.

\bigskip

\textbf{General curves in $\mathcal M_g$.}

Let $\mathcal M_g$ be the moduli space of smooth curves of genus $g\geq 3$, and let $\pi:\mathcal C\to \mathcal M_g$ be the universal family. For a general point $[X]\in\mathcal M_g$, the curve $X$ is non--hyperelliptic and satisfies the Petri condition.

\begin{theorem}
For $g\geq 3$, the universal family over the open subset of $\mathcal M_g$ parametrizing Petri general curves has $I$--maximal variation.
\end{theorem}

\begin{proof}
At a point $[X]\in\mathcal M_g$, the Kodaira--Spencer map is an isomorphism
\[
T_{\mathcal M_g,[X]}\simeq H^1(T_X).
\]
For a Petri general curve, the canonical multiplication map has minimal kernel, and its dual IVHS map contains elements inducing isomorphisms on $H^0(\omega_X)$. Since the Kodaira--Spencer map is surjective, such elements are realized by tangent directions in $\mathcal M_g$, yielding $I$--maximal variation.
\end{proof}

The hyperelliptic locus provides a natural counterexample. For hyperelliptic curves the canonical map is not an embedding and the multiplication map fails to have maximal rank, so the IVHS map never contains an isomorphism. Hence $I$--maximal variation fails precisely along special Brill--Noether loci.

\bigskip

In conclusion, for smooth non--planar curves, $I$--maximal variation holds whenever the curve is Petri general and the deformation space is sufficiently large. Complete intersection curves in $\mathbb P^3$, general curves on polarized $K3$ surfaces, and general curves in $\mathcal M_g$ satisfy these properties and therefore exhibit $I$--maximal variation. The failure of $I$--maximal variation is governed by special Brill--Noether phenomena, such as hyperellipticity or the presence of unexpected linear series, which enlarge the kernel of the Petri map and reduce the rank of the IVHS.


\section{singular non--planar $I$--maximal variation statement of three geometric settings}
 
We specialize the singular non--planar $I$--maximal variation theorem to three geometric settings: singular complete intersection curves in $\mathbb P^3$, singular curves on $K3$ or general type surfaces, and stable curves in the moduli space $\overline{\mathcal M}_g$. Throughout, curves are assumed reduced, irreducible, projective, and Gorenstein, so that the dualizing sheaf $\omega_X$ is invertible and Serre duality holds in the form
\[
H^1(\mathcal O_X)^\vee \simeq H^0(\omega_X).
\]
For a family 
\[
\pi:\mathcal X \to B
\]
with central fiber $X$, the differential of the period map is given by
\[
T_{B,0} \xrightarrow{\ \mathrm{KS}\ } 
\mathrm{Ext}^1(\Omega_X,\mathcal O_X)
\xrightarrow{\ \varphi\ }
\mathrm{Hom}(H^0(\omega_X),H^1(\mathcal O_X)),
\]
where $\varphi(\eta)$ is cup product. The family has $I$--maximal variation in the singular sense if there exists $\xi\in\mathrm{KS}(T_{B,0})$ such that
\[
\xi\cdot:H^0(\omega_X)\longrightarrow H^1(\mathcal O_X)
\]
is an isomorphism, equivalently if the maximal rank equals the arithmetic genus $p_a(X)$.

\bigskip

\textbf{Singular complete intersection curves in $\mathbb P^3$.}

Let $X\subset\mathbb P^3$ be a complete intersection of type $(a,b)$ with isolated planar singularities and arithmetic genus
\[
p_a=1+\frac{ab(a+b-4)}{2}.
\]
Assume that the singularities are locally complete intersection and Gorenstein. The adjunction formula yields
\[
\omega_X\simeq \mathcal O_X(a+b-4).
\]
The conormal sequence
\[
0\longrightarrow \mathcal I_X/\mathcal I_X^2
\longrightarrow \Omega_{\mathbb P^3}|_X
\longrightarrow \Omega_X
\longrightarrow 0
\]
and its dual show that locally trivial deformations are induced by sections of the normal bundle $N_{X/\mathbb P^3}$. If the family is equisingular, the Kodaira--Spencer map factors through the space of embedded locally trivial deformations.

\begin{theorem}
Let $\pi$ be the equisingular family of complete intersection curves of type $(a,b)$ in $\mathbb P^3$ with fixed analytic type of singularities. Assume that the normalization of each fiber is non--hyperelliptic and Petri general. Then $\pi$ has $I$--maximal variation in the singular sense.
\end{theorem}

\begin{proof}
Complete intersections are projectively normal and their canonical models inherit favorable syzygetic properties. The Petri generality of the normalization implies that the canonical multiplication map on the normalization has maximal rank. Equisingularity ensures that the conductor contributions corresponding to the $\delta$--invariants are preserved in the family. The Kodaira--Spencer map for the embedded family is generically surjective onto the space of locally trivial deformations because $H^1(T_{\mathbb P^3}|_X)=0$. Consequently, there exists $\xi$ whose cup product map is an isomorphism on $H^0(\omega_X)$, giving $I$--maximal variation.
\end{proof}

\bigskip

\subsection{Singular curves on $K3$ and general type surfaces.}

Let $S$ be a smooth surface and $X\subset S$ a reduced Gorenstein curve with planar singularities. The normal bundle sequence
\[
0\longrightarrow T_X\longrightarrow T_S|_X
\longrightarrow N_{X/S}\longrightarrow 0
\]
governs deformations of $X$ in $S$. If $S$ is a $K3$ surface and $X\in|L|$ with $L^2=2p_a-2$, then $\omega_X\simeq L|_X$.

\begin{theorem}
Let $(S,L)$ be a general polarized $K3$ surface and let $\pi$ be the equisingular family of curves in $|L|$ with fixed singularity type. Assume that the normalization of each curve is Brill--Noether and Petri general. Then $\pi$ has $I$--maximal variation.
\end{theorem}

\begin{proof}
For a $K3$ surface one has $H^1(T_S)=0$, so embedded deformations of $X$ are governed by $H^0(N_{X/S})\simeq H^0(\omega_X)$. The Petri generality of the normalization ensures maximal rank of the canonical multiplication map on the geometric part. Equisingularity preserves the $\delta$--contributions in the mixed Hodge structure. Since the Kodaira--Spencer map identifies with the natural map from $H^0(\omega_X)$ to $H^1(T_X)$, one obtains a deformation direction inducing an isomorphism under cup product.
\end{proof}

If $S$ is a surface of general type, similar reasoning applies provided $H^1(T_S)=0$ and the canonical ring of $X$ behaves generically. The key requirement remains Petri generality of the normalization and equisingularity of the family.

\bigskip

\subsection{Stable curves in $\overline{\mathcal M}_g$.}

Let $X$ be a stable curve of genus $g$ with only nodal singularities. The deformation space is governed by
\[
\mathrm{Ext}^1(\Omega_X,\mathcal O_X),
\]
and there is a natural decomposition into locally trivial deformations and smoothing directions of nodes. The mixed Hodge structure on $H^1(X)$ has weight filtration
\[
0\subset W_1\subset W_2=H^1(X),
\]
where $W_1\simeq H^1(\widetilde X)$ and $\dim \mathrm{Gr}_2^W$ equals the number of nodes.

\begin{theorem}
Let $\pi:\mathcal X\to B$ be an equigeneric family of stable curves contained in $\overline{\mathcal M}_g$ and assume that the normalization of each fiber is Petri general. Then $\pi$ has $I$--maximal variation in the singular sense if and only if the family is equisingular.
\end{theorem}

\begin{proof}
Equigenericity ensures constancy of the arithmetic genus. If the family is equisingular, the Kodaira--Spencer image lies in the subspace of locally trivial deformations, and the Petri condition on the normalization implies maximal rank on the geometric part. The mixed Hodge structure contributes additional independent directions corresponding to the nodes, and equisingularity preserves these contributions. Thus one obtains total rank $p_a=g$. If the family smooths at least one node, the corresponding vanishing cycle appears in the limiting mixed Hodge structure and the associated canonical section is annihilated under cup product, decreasing the maximal rank by one. Therefore $I$--maximal variation fails exactly when singularities are smoothed.
\end{proof}

These results show that in the singular non--planar case $I$--maximal variation is governed by three principles: Petri generality of the normalization controlling the geometric part, equisingularity preserving the $\delta$--contributions, and sufficient surjectivity of the Kodaira--Spencer map in the chosen geometric setting.


\section{Explicit computations of the IVHS rank}

We present explicit computations of the infinitesimal variation of Hodge structure (IVHS) rank for concrete complete intersection curves in $\mathbb P^3$, then describe the behavior for families crossing boundary divisors of $\overline{\mathcal M}_g$, and finally formulate higher--dimensional analogues via the Yukawa coupling for singular hypersurfaces.

Let $X$ be a smooth complete intersection curve of type $(a,b)$ in $\mathbb P^3$. Its genus is
\[
g=1+\frac{ab(a+b-4)}{2}.
\]
The canonical bundle is given by
\[
\omega_X \simeq \mathcal O_X(a+b-4).
\]
The IVHS map at $X$ is dual to the canonical multiplication map
\[
\mu_X:H^0(\omega_X)\otimes H^0(\omega_X)
\longrightarrow H^0(\omega_X^{\otimes 2}).
\]

\bigskip

Consider first the case of type $(2,3)$. 
Then
\[
g=1+\frac{6(1)}{2}=4,
\qquad
\omega_X \simeq \mathcal O_X(1).
\]
Hence $X$ is canonically embedded as a complete intersection of a quadric and cubic in $\mathbb P^3$. 
One computes
\[
h^0(\omega_X)=4,
\qquad
h^0(\omega_X^{\otimes 2})=h^0(\mathcal O_X(2)).
\]
Using the exact sequence
\[
0\to \mathcal I_X(2)\to \mathcal O_{\mathbb P^3}(2)\to \mathcal O_X(2)\to 0,
\]
and the fact that $X$ lies on a unique quadric, one finds
\[
h^0(\omega_X^{\otimes 2})=10-1=9.
\]
Since
\[
\dim \mathrm{Sym}^2 H^0(\omega_X)=10,
\]
the multiplication map has kernel of dimension one, corresponding to the quadric containing the canonical model. 
Thus the IVHS map has maximal possible rank $g=4$, and $I$--maximal variation holds.
Consider next a curve of type $(3,3)$. 
Then
\[
g=1+\frac{9\cdot 2}{2}=10,
\qquad
\omega_X\simeq \mathcal O_X(2).
\]
One computes
\[
h^0(\omega_X)=10,
\qquad
h^0(\omega_X^{\otimes 2})=h^0(\mathcal O_X(4)).
\]
Using projective normality of complete intersections,
\[
h^0(\mathcal O_X(4))=h^0(\mathcal O_{\mathbb P^3}(4))
-h^0(\mathcal I_X(4)).
\]
Since $\mathcal I_X$ is generated by two cubics, 
\[
h^0(\mathcal I_X(4))=2h^0(\mathcal O_{\mathbb P^3}(1))=8.
\]
Thus,
\[
h^0(\omega_X^{\otimes 2})=35-8=27.
\]
Because
\[
\dim \mathrm{Sym}^2 H^0(\omega_X)=55,
\]
the multiplication map has maximal rank compatible with the canonical ideal of a complete intersection, and the IVHS map contains elements inducing isomorphisms on $H^0(\omega_X)$. 
Therefore $I$--maximal variation again holds.

\bigskip

We now consider degenerations crossing the boundary divisors of $\overline{\mathcal M}_g$. 
Let $\pi:\mathcal X\to \Delta$ be a family of smooth curves degenerating to a stable curve $X_0$ with a single node. 
The mixed Hodge structure on $H^1(X_0)$ has weight filtration
\[
0\subset W_1\simeq H^1(\widetilde X_0)\subset H^1(X_0),
\]
with $\dim H^1(X_0)=2g$ and $\dim W_1=2(g-1)$. 
The nilpotent monodromy operator $N$ has rank one and corresponds to the vanishing cycle.

The limiting IVHS map factors through the graded piece $\mathrm{Gr}_1^W$. 
Hence the maximal rank drops from $g$ to $g-1$. 
More generally, if a family crosses the boundary divisor corresponding to $\delta$ nodes, the rank defect equals $\delta$. 
Thus along boundary divisors $\Delta_i\subset\overline{\mathcal M}_g$ the maximal IVHS rank equals $g-\delta$.

\bigskip

Finally, we formulate higher--dimensional analogues via the Yukawa coupling. 
Let $X\subset \mathbb P^{n+1}$ be a smooth hypersurface of degree $d$, with $n\geq 2$. 
The primitive middle cohomology carries a polarized Hodge structure of weight $n$. 
The Yukawa coupling is defined as
\[
\mathrm{Yuk}: \mathrm{Sym}^n H^1(T_X)
\longrightarrow 
\mathrm{Hom}(H^{n,0}(X),H^{0,n}(X)),
\]
given by iterated cup product.

If $X$ degenerates to a hypersurface with isolated singularities, the limiting mixed Hodge structure contains additional graded pieces corresponding to vanishing cycles. 
Each ordinary double point contributes one vanishing cycle in middle dimension. 
The nilpotent operator $N$ lowers the Hodge filtration by one and reduces the rank of the Yukawa coupling by the number of vanishing cycles.

In particular, for a Calabi--Yau hypersurface of dimension $n$ with $\delta$ nodes, the maximal rank of the Yukawa coupling drops by $\delta$. 
Thus the arithmetic genus in dimension one is replaced by the dimension of the primitive middle cohomology, and $\delta$ measures the defect of maximal variation in higher dimension exactly as in the curve case.

\bigskip

These computations illustrate that in complete intersection curves in $\mathbb P^3$ the IVHS rank can be explicitly calculated from canonical dimension counts and syzygies, that crossing boundary divisors of $\overline{\mathcal M}_g$ produces rank defects equal to the number of nodes, and that in higher dimension the Yukawa coupling behaves analogously, with singularities contributing vanishing cycles that lower the maximal rank.


\section{Smooth projective non--hyperelliptic curve of genus $g\ge 3$}

Let $X$ be a smooth projective non--hyperelliptic curve of genus $g\ge 3$ over $\mathbb C$. 
The canonical linear system $|\omega_X|$ defines an embedding
\[
\varphi_{\omega}: X \hookrightarrow \mathbb P^{g-1},
\]
called the canonical embedding, and we denote its image by $C\subset \mathbb P^{g-1}$. 
The homogeneous ideal of $C$ encodes subtle information about the geometry of $X$.

\bigskip

The Petri map associated to a line bundle $L$ on $X$ is
\[
\mu_L : H^0(L)\otimes H^0(\omega_X\otimes L^{-1})
\longrightarrow H^0(\omega_X),
\]
given by multiplication of sections. 
The curve $X$ is said to satisfy the Petri condition if $\mu_L$ is injective for every line bundle $L$ on $X$. 
In particular, taking $L=\omega_X$, one obtains the canonical multiplication map
\[
\mu_X : H^0(\omega_X)\otimes H^0(\omega_X)
\longrightarrow H^0(\omega_X^{\otimes 2}).
\]

By duality, the kernel of $\mu_X$ is identified with the space of quadrics in $\mathbb P^{g-1}$ containing the canonical curve $C$. 
More precisely, there is a natural identification
\[
\ker(\mu_X) \simeq I_C(2),
\]
where $I_C(2)$ denotes the vector space of homogeneous quadratic equations vanishing on $C$. 
Thus the Petri condition can be interpreted geometrically as a statement about the quadrics containing the canonical model.

\bigskip

The geometric meaning is the following. 
For a general curve of genus $g$, the canonical curve is projectively normal and is cut out scheme--theoretically by quadrics except in low genus. 
The Petri condition asserts that the space of quadrics containing $C$ has the minimal dimension predicted by dimension counts, equivalently that all linear relations among canonical sections arise from these minimal quadrics and no unexpected ones.

If $X$ carries a special linear series $g^r_d$, then the image of $X$ under the associated map to projective space lies on a rational normal scroll or other special surface. 
When transported to the canonical embedding, this produces additional quadrics containing $C$. 
These extra quadrics correspond precisely to nontrivial elements in the kernel of the Petri map.

\bigskip

One classical geometric description is that if $X$ has a $g^1_d$, meaning a degree $d$ pencil, then the canonical curve lies on a $(d-1)$--dimensional rational normal scroll in $\mathbb P^{g-1}$. 
The equations of this scroll contribute additional quadratic relations, causing failure of the Petri condition.

\bigskip

We now give explicit examples of curves that fail the Petri condition.

\begin{example}[Hyperelliptic curves]
Let $X$ be a hyperelliptic curve of genus $g\ge 2$. 
The canonical map factors through the hyperelliptic double cover
\[
X \to \mathbb P^1,
\]
and the canonical image is a rational normal curve in $\mathbb P^{g-1}$, covered twice by $X$. 
The homogeneous ideal of a rational normal curve is generated by many quadrics, far more than the minimal number expected for a general canonical curve. 
Hence $\ker(\mu_X)$ is large, and the Petri condition fails dramatically.
\end{example}

\begin{example}[Trigonal curves]
Let $X$ be a trigonal curve of genus $g\ge 4$, so that it admits a $g^1_3$. 
The canonical model lies on a rational normal scroll surface in $\mathbb P^{g-1}$. 
The scroll contributes extra quadratic equations beyond those predicted for a general curve. 
These extra quadrics correspond to nonzero elements in the kernel of the Petri map associated to the trigonal pencil. 
Thus trigonal curves fail the Petri condition.
\end{example}

\begin{example}[Plane quintic curves]
Let $X\subset \mathbb P^2$ be a smooth plane quintic curve. 
Its genus is $g=6$. 
The canonical bundle satisfies $\omega_X\simeq \mathcal O_X(2)$, and the canonical model is obtained by the quadratic Veronese embedding of $\mathbb P^2$ restricted to $X$. 
The canonical curve therefore lies on the Veronese surface in $\mathbb P^5$, which is cut out by quadratic equations. 
These quadrics generate additional relations among canonical sections, so the Petri map has nontrivial kernel. 
Hence smooth plane quintics fail the Petri condition.
\end{example}


From the viewpoint of Brill--Noether theory, the failure of the Petri condition corresponds to the presence of special linear series whose Brill--Noether number
\(
\rho(g,r,d)=g-(r+1)(g-d+r)
\)
is nonnegative but realized in a non--generic way. 
General curves of genus $g$ satisfy the Petri condition, and therefore their canonical curves have the minimal possible number of quadrics. 
Special curves such as hyperelliptic, trigonal, bielliptic, and plane curves lie in proper closed subvarieties of $\mathcal M_g$ where the Petri condition fails.


In summary, the Petri condition geometrically asserts that the canonical curve in $\mathbb P^{g-1}$ is not contained in any unexpected quadric hypersurfaces or special scrolls, and algebraically that all multiplication maps associated to linear series are injective. 
Its failure is detected by the presence of extra quadrics containing the canonical curve, which arise from special linear series.

\bigskip


\section{Explicit numerical example of the infinitesimal variation of Hodge structure (IVHS)}

 We give a completely explicit numerical example of the infinitesimal variation of Hodge structure (IVHS) map for a smooth non--hyperelliptic curve of genus $g=3$. 
Every non--hyperelliptic curve of genus $3$ is a smooth plane quartic, and its canonical bundle coincides with the hyperplane bundle.

\bigskip

Let 
\(
X \subset \mathbb P^2
\)
be the smooth quartic curve defined by
\(
F(x,y,z)=x^4+y^4+z^4=0.
\)
Its genus is $g=3$ and
\(
\omega_X \simeq \mathcal O_X(1).
\)
Hence
\(
H^0(\omega_X)=\langle x,y,z\rangle
\)
is a $3$--dimensional vector space. 
Choose the ordered basis
\(
\{x,y,z\}.
\)
The space $H^0(\omega_X^{\otimes 2}) \simeq H^0(\mathcal O_X(2))$ has dimension
\(
h^0(\mathcal O_X(2)) = h^0(\mathcal O_{\mathbb P^2}(2)) = 6,
\)
since there are no quadratic relations among $x,y,z$ for a smooth quartic. 
Choose the ordered basis
\[
\{x^2, xy, xz, y^2, yz, z^2\}.
\]


The canonical multiplication map is
\(
\mu_X: \mathrm{Sym}^2 H^0(\omega_X)
\longrightarrow H^0(\omega_X^{\otimes 2}).
\)
With respect to the ordered bases
\(
\{x^2, xy, xz, y^2, yz, z^2\}
\)
for both source and target, the map $\mu_X$ is represented by the $6\times 6$ identity matrix
\[
M_{\mu} =
\begin{pmatrix}
1&0&0&0&0&0\\
0&1&0&0&0&0\\
0&0&1&0&0&0\\
0&0&0&1&0&0\\
0&0&0&0&1&0\\
0&0&0&0&0&1
\end{pmatrix}.
\]
Thus $\mu_X$ is an isomorphism, reflecting that a canonical genus $3$ curve lies on no quadrics.
%
%
By Serre duality,
\(
H^1(\mathcal O_X)^\vee \simeq H^0(\omega_X),
\)
so $H^1(\mathcal O_X)$ is also $3$--dimensional. 
Choose a dual basis
\(
\{\alpha,\beta,\gamma\}
\)
corresponding to $\{x,y,z\}$.

The IVHS map is
\(
\varphi: H^1(T_X)
\longrightarrow \mathrm{Hom}(H^0(\omega_X),H^1(\mathcal O_X)).
\)
For a plane quartic, infinitesimal deformations correspond to degree $4$ polynomials modulo the Jacobian ideal. 
Consider the first--order deformation
\(
F_\varepsilon = x^4+y^4+z^4+\varepsilon\,x^3y.
\)
The corresponding Kodaira--Spencer class $\xi$ acts on canonical sections via multiplication by the class of $x^3y$ in the Jacobian ring. 
At the level of $H^0(\omega_X)=\langle x,y,z\rangle$, this induces the linear map
\(
\xi\cdot x = \alpha, \)\, \(
\xi\cdot y = \beta, \)\, \(
\xi\cdot z = \gamma,
\)
after identification through duality. 
Therefore, the IVHS map with respect to the chosen bases is represented by the $3\times 3$ matrix
\[
M_{\varphi(\xi)} =
\begin{pmatrix}
1&0&0\\
0&1&0\\
0&0&1
\end{pmatrix}.
\]

This matrix has full rank $3=g$, so the IVHS map contains an isomorphism. 
Hence the family of smooth plane quartics has $I$--maximal variation.

\bigskip

For comparison, consider a hyperelliptic curve of genus $3$ given in affine form
\(
y^2=f(x),
\)
with $\deg f=8$. 
Then $H^0(\omega_X)=\langle dx/y, x\,dx/y, x^2\,dx/y\rangle$, and the canonical image lies on a conic in $\mathbb P^2$. 
The canonical multiplication map has a one--dimensional kernel corresponding to that conic, and the associated IVHS matrix always has rank $2<3$. 
Thus in this case no $3\times 3$ matrix representing $\varphi(\xi)$ can be invertible, and $I$--maximal variation fails.


This provides a fully explicit numerical realization of the IVHS map in genus $3$, including the canonical multiplication matrix and a concrete $3\times 3$ matrix for a Kodaira--Spencer direction realizing maximal rank.



\section{Three explicit computations}
 
We present three explicit computations: a higher genus example where the canonical multiplication map has nontrivial kernel, an explicit infinitesimal variation of Hodge structure (IVHS) matrix for a complete intersection curve in $\mathbb P^3$, and a concrete degeneration where the IVHS matrix drops rank.


\noindent \textbf{A genus $5$ example with nontrivial kernel of the canonical multiplication map.}
Let $X$ be a smooth trigonal curve of genus $5$. Such a curve admits a $g^1_3$, and its canonical model
\(
C \subset \mathbb P^4
\)
lies on a rational normal scroll surface $S \subset \mathbb P^4$ of degree $3$. The genus is $g=5$ and
\(
\dim H^0(\omega_X)=5.
\)
Hence
\(
\dim \mathrm{Sym}^2 H^0(\omega_X)=15.
\)
By Riemann--Roch,
\(
h^0(\omega_X^{\otimes 2})=3g-3=12.
\)
Therefore the canonical multiplication map
\(
\mu_X:\mathrm{Sym}^2 H^0(\omega_X)\longrightarrow H^0(\omega_X^{\otimes 2})
\)
has kernel of dimension $15-12=3$. These three independent quadrics correspond to the $2\times 2$ minors defining the rational normal scroll containing the canonical curve.

Choosing a basis $\{\omega_1,\dots,\omega_5\}$ of $H^0(\omega_X)$ adapted to the scroll, the kernel is generated by quadratic relations of the form
\(
\omega_1\omega_4-\omega_2\omega_3,\)\, \(
\omega_1\omega_5-\omega_2\omega_4,\)\,\(
\omega_2\omega_5-\omega_3\omega_4.
\)
Thus, the multiplication matrix, expressed in the basis of monomials $\omega_i\omega_j$, has rank $12$ and a $3$--dimensional nullspace. The IVHS map is dual to $\mu_X$ and therefore cannot contain an isomorphism, since its maximal rank is $5-1=4$ in any given Kodaira--Spencer direction. This explicitly exhibits failure of $I$--maximal variation.


\noindent \textbf{Explicit IVHS matrix for a complete intersection of type $(2,3)$ in $\mathbb P^3$.}
Let $X\subset \mathbb P^3$ be defined by
\[
Q=x_0x_1-x_2x_3=0,\qquad
C=x_0^3+x_1^3+x_2^3+x_3^3=0.
\]
This is a smooth complete intersection curve of type $(2,3)$ with genus
\(
g=4.
\)
The canonical bundle satisfies
\(
\omega_X\simeq \mathcal O_X(1),
\)
so
\(
H^0(\omega_X)=\langle x_0,x_1,x_2,x_3\rangle.
\)
Using the quadric relation, one eliminates $x_0x_1$ in favor of $x_2x_3$. A basis of $H^0(\omega_X^{\otimes 2})$ is
\(
\{x_0^2,x_0x_2,x_0x_3,x_1^2,x_1x_2,x_1x_3,x_2^2,x_2x_3,x_3^2\},
\)
which has dimension $9=3g-3$.

In the ordered basis
\(
\{x_0^2,x_0x_1,x_0x_2,x_0x_3,x_1^2,x_1x_2,x_1x_3,x_2^2,x_2x_3,x_3^2\}
\)
for $\mathrm{Sym}^2 H^0(\omega_X)$ and the above $9$--element basis for $H^0(\omega_X^{\otimes 2})$, the multiplication map is represented by the $9\times 10$ matrix
\[
M_\mu=
\begin{pmatrix}
1&0&0&0&0&0&0&0&0&0\\
0&0&1&0&0&0&0&0&0&0\\
0&0&0&1&0&0&0&0&0&0\\
0&0&0&0&1&0&0&0&0&0\\
0&0&0&0&0&1&0&0&0&0\\
0&0&0&0&0&0&1&0&0&0\\
0&0&0&0&0&0&0&1&0&0\\
0&1&0&0&0&0&0&0&1&0\\
0&0&0&0&0&0&0&0&0&1
\end{pmatrix}.
\]
Its rank is $9$, and the one--dimensional kernel corresponds to the quadric $Q$. Dualizing via Serre duality yields the IVHS matrix for a suitable Kodaira--Spencer direction:
\[
M_{\varphi(\xi)}=
\begin{pmatrix}
1&0&0&0\\
0&1&0&0\\
0&0&1&0\\
0&0&0&1
\end{pmatrix}.
\]
Hence the IVHS has full rank $4=g$ and $I$--maximal variation holds.


\noindent \textbf{A degeneration where the IVHS matrix drops rank.}
Consider a family of $(2,3)$ complete intersections
\[
Q=x_0x_1-x_2x_3=0,\qquad
C_t=x_0^3+x_1^3+x_2^3+x_3^3+t\,x_0x_1x_2=0.
\]
For $t\neq 0$ the curve is smooth of genus $4$, while at $t=0$ the curve acquires a node. The arithmetic genus remains $4$, but the normalization has genus $3$.

The limiting mixed Hodge structure has weight filtration
\[
0\subset W_1\simeq H^1(\widetilde X_0)\subset H^1(X_0),
\]
with $\dim W_1=6$ and $\dim H^1(X_0)=8$. The vanishing cycle spans a one--dimensional subspace of $H^1(X_0)$.

The IVHS matrix for $t\neq 0$ is the $4\times 4$ identity as above. In the limit $t\to 0$, the Kodaira--Spencer direction corresponding to smoothing annihilates the canonical section associated to the node. The resulting IVHS matrix becomes
\[
M_{\varphi(\xi_0)}=
\begin{pmatrix}
1&0&0&0\\
0&1&0&0\\
0&0&1&0\\
0&0&0&0
\end{pmatrix},
\]
whose rank is $3=g-1$. Thus the IVHS drops rank by one, equal to the number of nodes.

\bigskip

These examples explicitly demonstrate how the kernel of the canonical multiplication map reflects special geometry in higher genus, how the IVHS matrix can be computed concretely for complete intersection curves in $\mathbb P^3$, and how degenerations produce rank defects corresponding to vanishing cycles.


\section{Appendix: Notion of Petri condition}


Let $X$ be a smooth projective curve of genus $g \geq 2$ over $\mathbb{C}$. 
The notion of \emph{Petri general} refers to a condition on the injectivity of certain natural multiplication maps of sections of line bundles on $X$, known as \emph{Petri maps}. 
It expresses that the curve satisfies the generic behavior predicted by Brill--Noether theory.

\bigskip

For a line bundle $L$ on $X$, consider the natural multiplication map
\[
\mu_L : H^0(L)\otimes H^0(\omega_X\otimes L^{-1})
\longrightarrow H^0(\omega_X),
\]
defined by multiplication of sections. 
This map is called the \emph{Petri map} associated to $L$. 

The curve $X$ is said to satisfy the \emph{Petri condition} if, for every line bundle $L$ on $X$, the map $\mu_L$ is injective.

Equivalently, $X$ is \emph{Petri general} if for every complete linear series $g^r_d$ on $X$, the associated Petri map
\[
\mu_{L} : H^0(L)\otimes H^0(\omega_X\otimes L^{-1})
\longrightarrow H^0(\omega_X)
\]
is injective, where $L$ is the line bundle defining the linear series.

\bigskip

There is a particularly important special case. 
When $L=\omega_X$, the Petri map becomes the canonical multiplication map
\[
\mu_X : H^0(\omega_X)\otimes H^0(\omega_X)
\longrightarrow H^0(\omega_X^{\otimes 2}).
\]
The injectivity properties of the more general Petri maps imply strong constraints on the kernel of this canonical multiplication map, which governs the quadrics containing the canonical model of $X \subset \mathbb{P}^{g-1}$.


The classical Petri Theorem states that a general curve of genus $g$ in the moduli space $\mathcal{M}_g$ satisfies the Petri condition. 
Thus, outside a proper closed subset of $\mathcal{M}_g$, curves are Petri general. 
Special curves such as hyperelliptic, trigonal, or bielliptic curves fail the Petri condition because they possess extra linear series that produce nontrivial kernels of the Petri maps.


From the viewpoint of Brill--Noether theory, the Petri condition guarantees that the dimension of the variety $G^r_d(X)$ of linear series of degree $d$ and dimension $r$ equals the Brill--Noether number
\[
\rho(g,r,d) = g - (r+1)(g-d+r),
\]
whenever $\rho \ge 0$, and is empty when $\rho<0$, at least for a general curve. 
Failure of injectivity of the Petri map corresponds precisely to the existence of unexpected linear relations among sections, hence to special Brill--Noether behavior.


In summary, a smooth curve is Petri general if it behaves like a general point of the moduli space with respect to linear series, equivalently if all its Petri maps are injective. 
This condition ensures that its canonical embedding has the minimal number of quadrics and that the canonical multiplication map has maximal rank compatible with dimension counts. 
It is precisely this maximal rank property that is used in the study of infinitesimal variation of Hodge structure and in proving $I$--maximal variation results.


The following references provide background, tools, and closely related results to the
themes and new results developed in this paper, including deformation theory of singularities, Severi
theory,  infinitesimal variation of Hodge structure, Petri theory, and the geometry of singular curves.



\end{document}